\documentclass{amsart}
\usepackage{amssymb,graphicx,float}
\usepackage{pinlabel}
\usepackage{palatino}
\usepackage{longtable}
\usepackage[dvipsnames,svgnames,table]{xcolor}
\usepackage[textwidth=6.5in,textheight=8.5in,left=1in,top=1in]{geometry}
\usepackage{thmtools}
\usepackage{thm-restate}

\newtheorem{theorem}{Theorem}
\newtheorem{lemma}[theorem]{Lemma}
\newtheorem{prop}[theorem]{Proposition}
\newtheorem{corollary}[theorem]{Corollary}

\theoremstyle{definition}
\newtheorem{example}{Example}
\newtheorem{definition}{Definition}
\newtheorem{remark}{Remark}

\usepackage[hyperfootnotes=false]{hyperref}
\hypersetup{
  colorlinks,
  citecolor=Purple,
  linkcolor=Black,
  urlcolor=Green}
\begin{document}

\title{Ascending, Bridge, and Unknotting Numbers of Knotted Graphs}

\author[Acu]{Bahar Acu}
\address{Department of Mathematics \\ Claremont Colleges: Pitzer College \\ Claremont \\ CA \\ USA}
\email{baharacu@gmail.com}

\author[Pongtanapaisan]{Puttipong Pongtanapaisan}
\address{Department of Mathematics \\ Claremont Colleges: Pitzer College \\ Claremont \\ CA \\ USA}
\email{puttip@pitzer.edu}

\author[Singhal]{Ritikaa Singhal}
\email{ritikaa.singhal@gmail.com}


\begin{abstract}
In this paper we define the ascending number and bridge number for virtual and welded spatial graphs. We prove that the former admits a lower bound in terms of the latter. This relationship allows us to compare different notions of unknotting virtual spatial graphs by crossing changes.
\end{abstract}

\maketitle

\section{Introduction}
Topological methods have long been used to study entanglement in biological systems. In DNA topology, knots and links record the topological state of closed DNA molecules, and local strand-passage events provide a mathematical model for the action of enzymes such as topoisomerases and recombinases. Tangle-theoretic models of site-specific recombination and grid-diagram models of topoisomerase-mediated simplification make this relationship especially explicit \cite{ErnstSumners1990,BarbensiEtAl2020}. Topological complexity also appears in proteins, where knots, links, and related spatial structures arise in the study of protein conformations and folding pathways \cite{FlapanHeller2015}. These examples motivate the systematic study of local operations, such as crossing changes, and of the numerical invariants that measure the complexity of simplifying a diagram by such operations.\smallskip

In classical knot theory, crossing changes give a standard algorithm for unlinking. Given a based, ordered, oriented diagram, one traverses the components in order and changes precisely those crossings that are first encountered as under-crossings. The resulting diagram is descending, and the classical ``roller-coaster trick'' shows that every descending classical diagram represents the unlink \cite{Adams2004,Ozawa2010}. Ozawa formalized this construction through the ascending number \(a(L)\), the minimum number of crossing changes needed to transform a based, ordered, oriented diagram of \(L\) into a descending diagram \cite{Ozawa2010}. Thus, for classical links, the ascending number gives an upper bound for the unlinking number. Related diagrammatic quantities appear in the work of Shimizu under the name warping degree \cite{Shimizu2010,Shimizu2011}. \smallskip

Classical spatial graph theory extends knot theory by replacing embedded circles with embedded graphs in $S^3$ or $\mathbb{R}^3$. Virtual spatial graphs provide a further extension of this viewpoint. Virtual spatial graph theory combines Kauffman's virtual knot theory with spatial graph theory, and it can be interpreted topologically in terms of graphs embedded in thickened surfaces, considered up to stabilization \cite{Kauffman1999,FlemingMellor2007,Kuperberg2003,DengJinKauffman2022}. This perspective is useful not only as a diagrammatic extension of classical spatial graph theory, but also for studying knotted graph-like objects with open ends arising in models of proteins \cite{GugumcuKauffmanPongtanapaisan2024}. Welded link theory, obtained by allowing the over-forbidden move, has a different but equally important topological interpretation: Satoh's Tube map relates welded links to ribbon torus-knots in four-space, and recent work extends this construction from welded graphs to ribbon surface-links \cite{Satoh2000,AudouxMeilhanYasuhara2025}. These connections motivate studying diagrammatic invariants that apply uniformly to classical, virtual spatial graphs.\smallskip

In the welded category, the classical roller-coaster mechanism survives in the one-component case: Satoh proved that crossing changes form an unknotting operation for welded knots \cite{Satoh2018}. Equivalently, a descending welded knot diagram represents the welded unknot. This statement, however, should not be interpreted as a statement about all virtual spatial graphs. Already the virtual Hopf link gives a descending two-component welded link diagram, after ordering the overpassing component first, which is not welded equivalent to the unlink. This is because its ordered virtual linking numbers are $v\ell k(J,K)=1$ and $v\ell k(K,J)=0$ whereas both ordered virtual linking numbers vanish for the unlink. Since ordered virtual linking numbers are welded invariants, this gives a descending welded link which is not the welded unlink. In the virtual category, descending virtual knot and link diagrams also need not represent trivial objects, so the ascending number is not simply an unlinking algorithm. Nevertheless, descending virtual diagrams still form a special and comparatively tractable class.\smallskip

The goal of this paper is to demonstrate that a naive way of unknotting may not be as efficient as possible. We accomplish this by bounding the ascending number from below in terms of the bridge number, which is an invariant we discuss in Section \ref{sec:prelims}.

\begin{restatable}{theorem}{mainbridgebound} \label{thm:mainbridgebound} Let \(G\) be an \(n\)-component virtual or welded spatial graph, \(a(G)\) denote the ascending number of \(G\), and \(bridge(G)\) denote its bridge number. Then \[ a(G) \geq bridge(G)-n. \] \end{restatable}

It will be evident that interpreting the bridge number in terms of diagram coloring is useful \cite{BlairKjuchukovaVelazquezVillanueva2020,Pongtanapaisan2019,blackwell2026bridge}. The following theorem gives examples showing that unknotting a spatial graph by first making its diagram descending can be highly inefficient. We focus on almost unknotted spatial graphs, meaning that every proper subgraph is trivial. This condition makes the result more significant because spatial graphs with large bridge number can otherwise be constructed in a relatively simple way. For example, one may begin with a nontrivial knot diagram $D$ and add an edge whose interior is disjoint from $D$, thereby obtaining a $\theta$-curve containing the original knot as a constituent knot. The $\theta$-curve $3_1$ in Moriuchi's table arises in this way. Consequently, a straightforward method for producing a $\theta$-curve with large bridge number is to start with a knot of large bridge number and add such an edge. By contrast, we are interested in examples whose complexity is genuinely global and cannot be attributed to a single complicated proper subgraph.

\begin{restatable}{theorem}{arbitrarilylarge} \label{thm:arbitrarilylarge}
\begin{enumerate}
    \item[(a)] For any nonnegative integer $n$, there exists an almost
    unknotted $\theta$-curve $G$ satisfying
    \[
        a(G)=2n
        \qquad\text{and}\qquad
        a(G)-u(G)\ge n.
    \]

    \item[(b)] There exists a family of almost unknotted handcuff graphs
    $\{H_n\}_{n\in\mathbb N}$ with
    \[
        u(H_n)=1
        \qquad\text{and}\qquad
    \lim_{n\rightarrow \infty} a(H_n)-u(H_n) =\infty.
    \]
\end{enumerate}
\end{restatable}

The examples so far in the theorem above are all classical. Our primary focus will be on examples with virtual crossings in the final section. It is not hard to see that allowing classical crossings to be replaced by virtual crossings gives an unknotting operation for virtual knots. In Figure \ref{fig:vircross}, this corresponds to replacing either of the classical crossings on the left or right by the virtual crossing in the middle, or conversely. However, crossing changes alone do not unknot every virtual knot. Thus, one can ask how far a virtual knot is from being unknotted by crossing changes alone, allowing virtualizations only when crossing changes are insufficient. This idea is formalized by Kaur et al. as the \textit{unknotting index} \cite{kaur2019unknotting}. The invariant is a lexicographically ordered tuple, where the number of virtualizations is minimized first, and the number of crossing changes is minimized subject to that choice.\smallskip

There is a relationship between ascending number and the unknotting index. For virtual links, such a relationship is shown in Theorem 2.3 of \cite{kaur2019unknotting}, which we restate here.

\begin{theorem}[\cite{kaur2019unknotting}]
    \label{thm:theta} Let $D$ be a diagram representing a virtual link $L$ with $v$ virtual crossings. Then \[ u(G)\leq (v,a(L)).\]

\end{theorem}

This motivates a broader comparison between two measures of complexity: the ascending number, which records how far a diagram is from being descending, and the unknotting number, which records how hard it is to reduce the knot to the unknot by crossing changes. We conclude by presenting computational results for these invariants on several virtual knots from Green’s table.

\subsection*{Organization of the paper}

The paper is organized as follows. In Section~\ref{sec:prelims}, we recall basic definitions related to virtual spatial graphs. In Section~\ref{sec:bridgenumber}, we discuss various notions of bridge numbers for different knotted objects. In Section~\ref{sec:ascend}, we show that the ascending number is well-defined in the virtual/welded category.

\subsection*{Acknowledgments}
The authors would like to thank Pitzer College for funding through the Research \& Awards Program.
\section{Virtual spatial graphs}
\label{sec:prelims}

In 1999, Louis H. Kauffman introduced \emph{virtual knots} to the field of knot theory as a generalization of classical knots \cite{Kauffman1999}. Later on, Fleming and Mellor introduced virtual spatial graphs, which can be thought of as graph drawings that have classical crossings, but now with the addition of what we call a \emph{virtual crossing} \cite{FlemingMellor2007}.

\begin{definition}
    A \emph{virtual crossing} is a 4-valent vertex with
a small circle around it that does not possess any over or under crossing information (Figure~\ref{fig:vircross}). A \textit{virtual spatial graph} is a graph diagram where crossings are decorated either as a classical crossing or virtual crossing modulo planar isotopies and Reidemeister-type moves (I)-(VI),(I*)-(V*) shown in Figure \ref{fig:Reid}. 
\end{definition}

\begin{figure}[!ht]
    \centering
    \includegraphics[width=.3\textwidth]{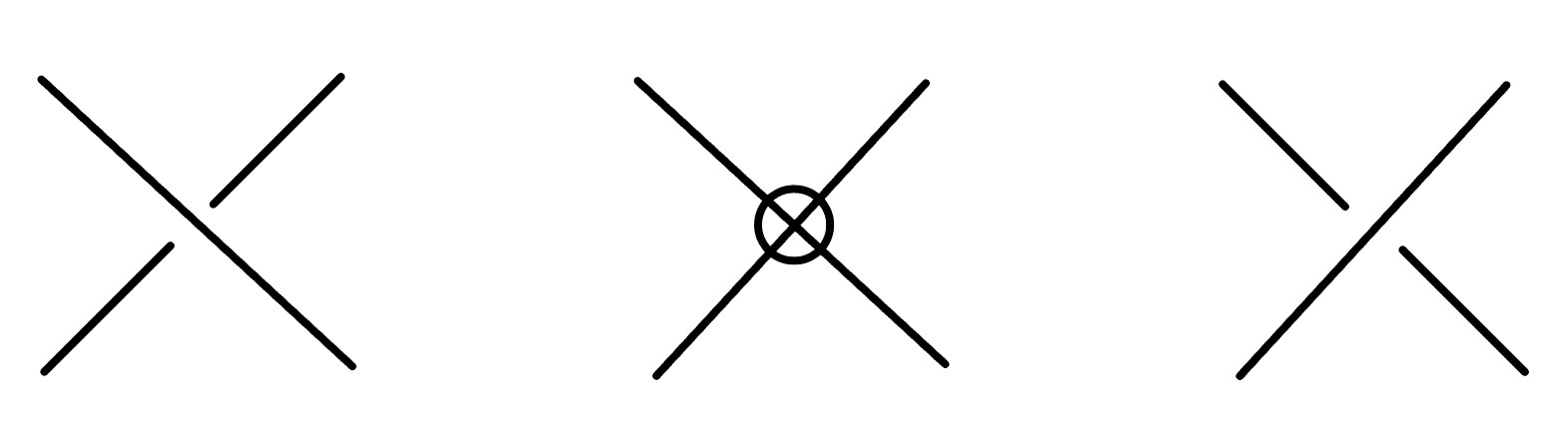}
    \caption{Virtual crossing (center) and classical crossings.}
    \label{fig:vircross}
\end{figure}

\begin{figure}[H]
    \centering\includegraphics[width=.6\textwidth]{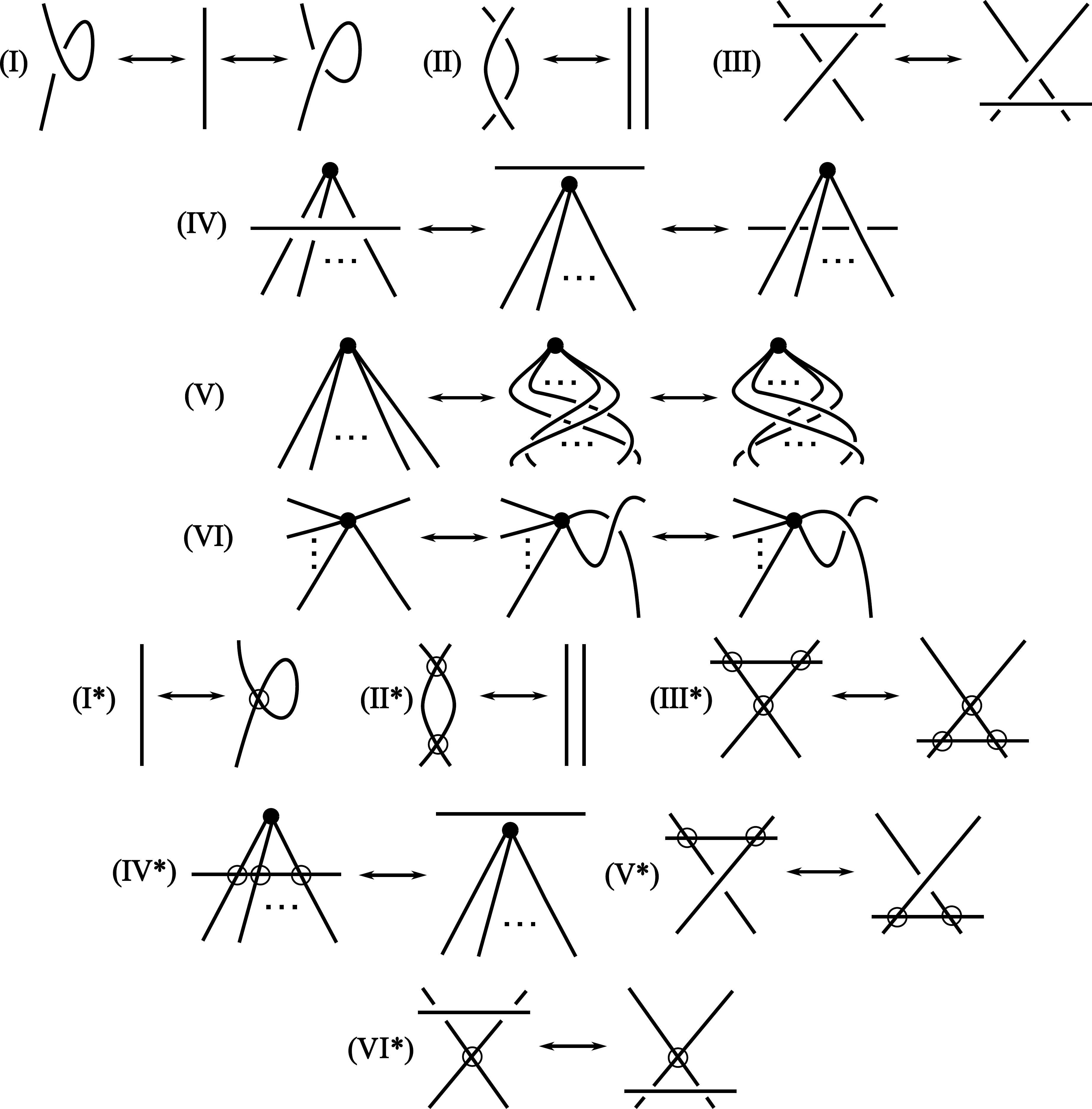}
    \caption{Reidemeister moves involved in virtual spatial graphs and welded links.}
    \label{fig:Reid}
\end{figure}

\begin{figure}[!ht]
    \centering\includegraphics[width=.7\textwidth]{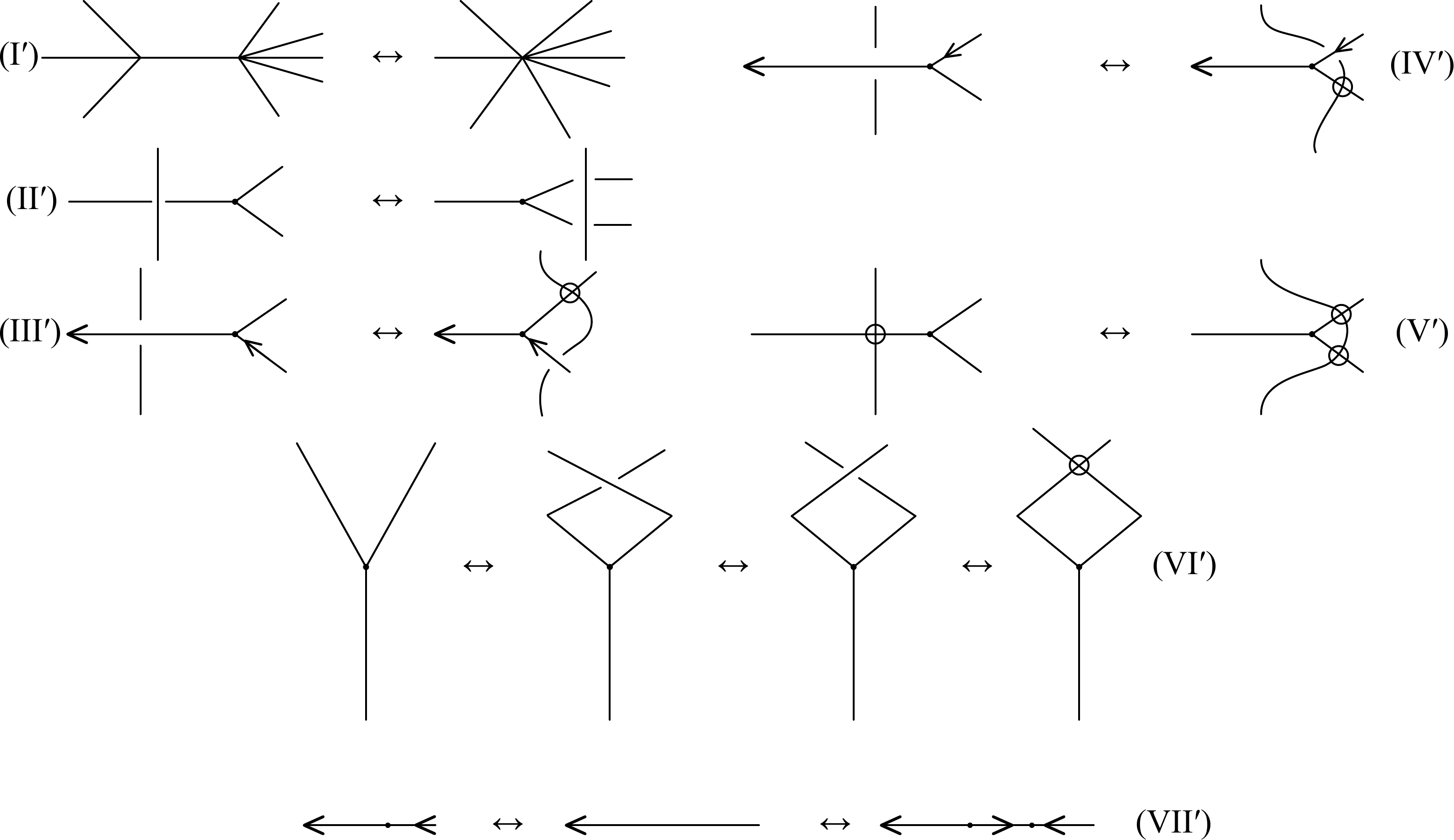}
    \caption{Reidemeister moves for welded graphs.}
    \label{fig:WeldReid}
\end{figure}

 The virtual crossing is not part of the structure of the knotted graph itself; it is an artifact of projecting a graph in a thickened surface to a plane. \smallskip

 When the abstract graph type is a collection of circles, we refer to such spatial graphs as virtual/welded links. When moves with the * symbols are not considered and virtual crossings are not present, we recover classical link theory and classical spatial graph theory. A virtual link, where we include move (VI*) into the equivalence relation is called a \textit{welded link}.\smallskip

 For welded spatial graphs, where the abstract graph type is not a collection of circles, the story is slightly different. The set of moves relating equivalent welded graphs is given in \cite{kauffman2008invariants}, which is also reproduced in Figure \ref{fig:WeldReid} for convenience. The readers are encouraged to consult Figure 46 of \cite{kauffman2008invariants} to see why moving a strand under a vertex introduces a virtual crossing. The motivation is to model the theory of ribbon surfaces in 4-space by broken sheet diagrams. The connection to knotted surface theory is through Satoh's Tube map \cite{kauffman2008invariants}. Roughly, there is a way to inflate a virtual spatial graph diagram $D$ to get a surface diagram $Tube(D)$ that is a double of a ribbon surface with boundary. It's also well-known that the fundamental group of the exterior of the surface with diagram $Tube(D)$ can be computed by using the classical Wirtinger algorithm on $D$, where virtual crossings do not contribute new relations. Furthermore, Satoh showed that if $K$ is welded equivalent to $K'$, then the corresponding ribbon tori are ambient isotopic. \smallskip

 The following subsection will be relevant to our computational section. It makes sense to discuss Gauss codes for virtual spatial graphs in general, but the computational work in this paper only concerns knots.
 
 \subsection{Gauss Codes for Virtual Knots}
%===============================================================

The \textit{Gauss code} of a virtual knot with $m>0$ classical crossings consists of the following symbols:
    \begin{itemize}
        \item Integers ranging from $\{1,...,m\}$, indexing the crossing,
        \item The letter $O$, denoting an overcrossing,
        \item The letter $U$, denoting an undercrossing,
        \item A sign ($+$ or $-$) indicating the handedness (positive or negative crossing).
\end{itemize}

 While not every Gauss code can be realized by a classical knot diagram, every Gauss code corresponds to a virtual knot. This makes virtual knot theory a natural setting for studying knots from a purely combinatorial perspective.

Similar to the original Gauss code, each crossing index appears exactly twice in the code, however now more specifically once as an over crossing and once as an under crossing. In the virtual Gauss code, clusters of 3 symbols appear together; one $O/U$, one index symbol, one $+/-$. These 3-symbol clusters give  the recipe to create each segment of the necessary diagram which can be constructed in the following way. \smallskip

Notice the first 3-symbol cluster, based on the handedness, draw the first crossing with the respective orientation. Label it with the given index and continue drawing out the under or over thread based on the specification. Then, illustrate the next 3-symbol cluster by continuing the line. Draw its respective under or over crossing, orienting the other strand with respect to the handedness of the crossing. Ensure to label all crossings as you go. Continue this process until all crossings are successfully illustrated with the correct respective handedness. Finally connect the line back to the starting point. \smallskip

In this process, virtual crossings, crossings that are not explicitly written in the Gauss code, will be created. Ensure to indicate those correctly by illustrating a circle around those specific vertices. \smallskip

Consider the example in Figure~\ref{fig:construction} of how to illustrate a knot diagram from a given Gauss code for the virtual knot 2.1.\smallskip

It is important to note that the algebraic encoding of the Gauss code does not always uniquely determine the diagram of a given virtual knot, and some geometric features may not be immediately apparent from the code alone. Figure~\ref{fig:varproject}, for example, shows how the same Gauss code can have varying projections and different placements or quantities of virtual crossings. 

\begin{figure}[H]
    \centering
    \includegraphics[width=0.3\linewidth]{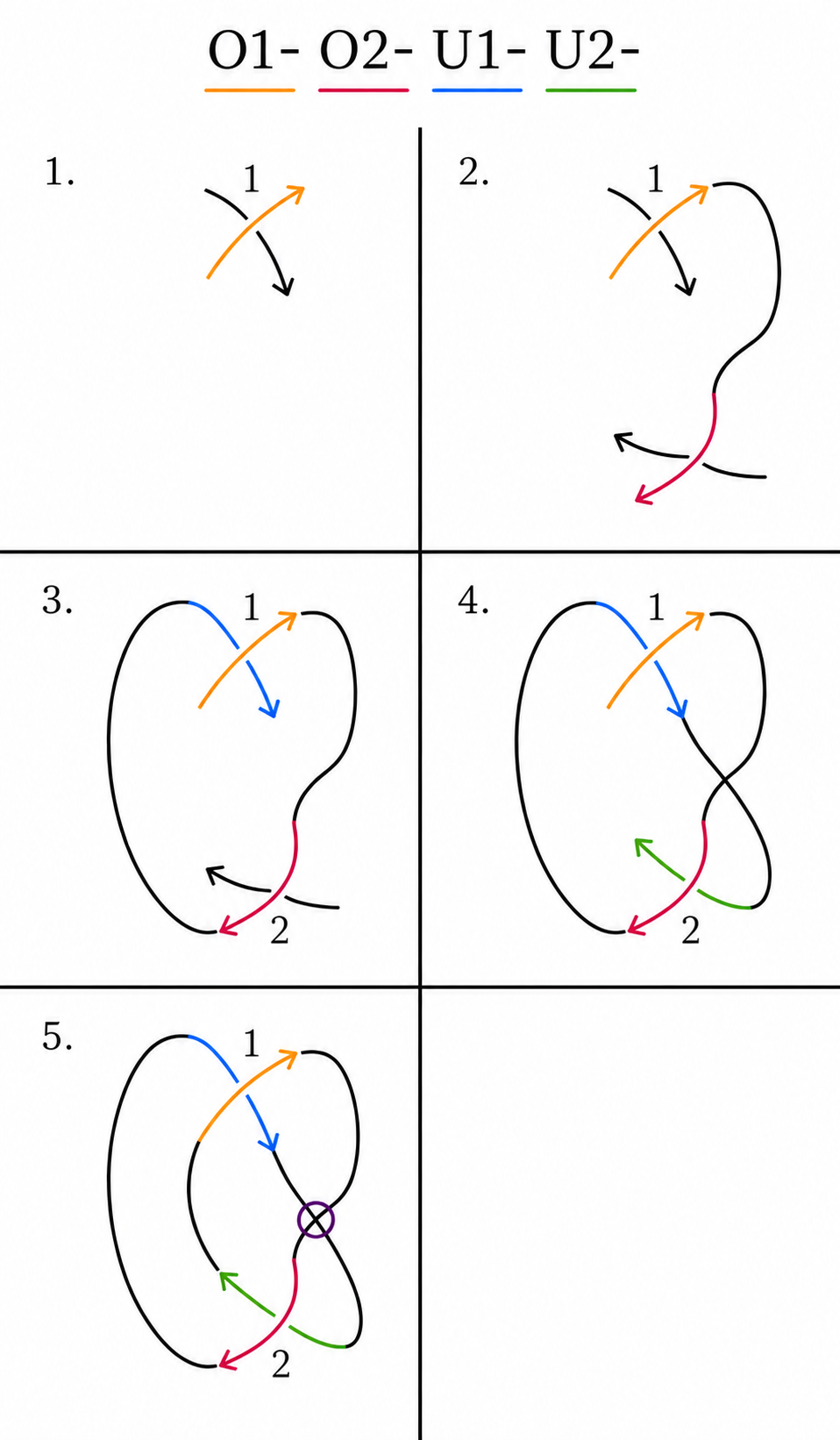}
    \caption{Step-by-step construction of virtual knot 2.1.}
    \label{fig:construction}
\end{figure}

\begin{figure}[H]
    \centering
    \includegraphics[width=0.3\linewidth]{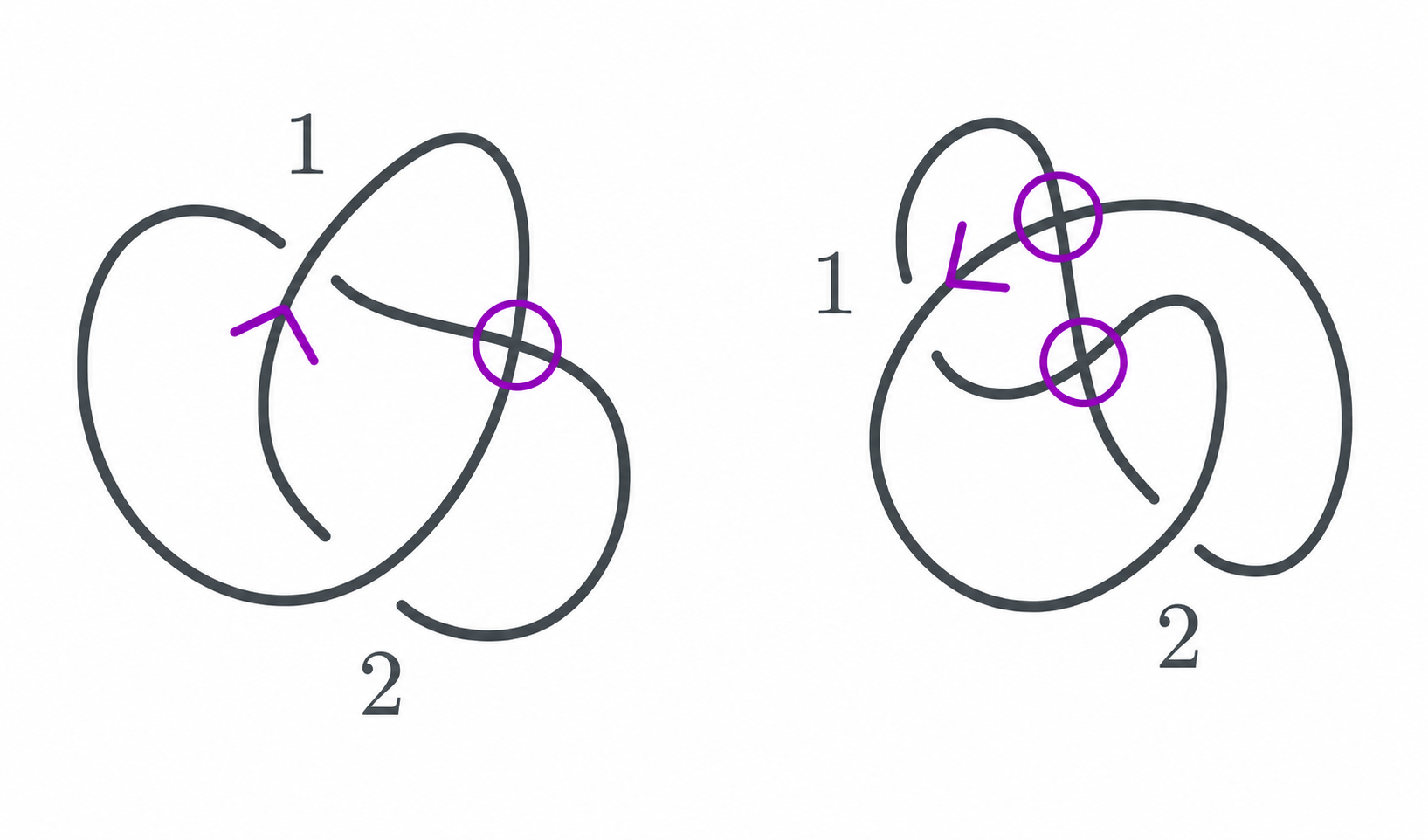}
    \caption{Varying projections of virtual knot 2.1.}
    \label{fig:varproject}
    \end{figure}
\section{Bridge number}\label{sec:bridgenumber}

\subsection{Bridge number of spatial graphs}
There are many ways to define the bridge number of classical spatial graphs. The following formulation appeared in \cite{ozawa2012bridge}. \smallskip

A \textit{bridge decomposition} of a spatial graph $G$ in the 3-sphere $S^3$ is a decomposition $(S^3,G) = (B^3_+,T_+)\cup_{\Sigma} (B^3_{-},T_{-})$ such that each component of $T_{\pm}$ is a tree or an arc, and there exists a disk $D_{\pm}$ in $B_{\pm}$
containing $T_{\pm}$. Each bridge position comes equipped with a surface $\Sigma$ called the \textit{bridge sphere}, which is a 2-sphere containing the endpoints of $T_{\pm}$. The (geometric) \textit{bridge number} of $G$ to be the minimum number of intersection points $\dfrac{1}{2}|G\cap \Sigma|$ over all bridge spheres $\Sigma$ for $G$. \smallskip

To be more compatible with the notion of ascending number, we may impose restrictions on the types of trees that can appear in $B_+^3$ or $B_-^3$. We may also count the number of components in each trivial tangle to define bridge number $bridge(G)$ instead of Ozawa's way of counting the number of intersection points. Curious readers can read ahead to see that a maximal tree $T$ plays a role in the definition of ascending numbers, so it is appropriate to make $T$ and its neighboring edges lie entirely in $B_+^3$ or $B_-^3$. Whenever we would like to restrict the types of trees, we will decorate $bridge(G)$ with a subscript. For instance, when we compute $bridge_{\theta}(G)$ for a $\theta$-curve $G$, we require that two vertices lie on the same side, say $B_+^3$ and we count the number of arc components.

\subsection{Bridge number of virtual and welded spatial graphs}

In the virtual and welded setting, the ambient space is a bit more cumbersome to work with, so we turn to the diagrammatic perspective.\smallskip

An \textit{overpass} is a connected component in a spatial graph diagram that contains at least one overcrossing and contains no undercrossings. We define the bridge number of a virtual spatial graph to be the minimum number of overpasses over all diagrams $D$ representing $G$. We also require that each connected component has at least one overpass. We can achieve this without changing the spatial graph type by performing Reidemeister one moves, for instance. Otherwise, one can have a diagram of the unlink with bridge number one, which seems unnatural.\smallskip

In fact, for classical spatial graphs, one can unify the concepts and show that the definition in terms of trees and overpasses are equivalent (see Corollary 3.3 of \cite{blackwell2026bridge}). 

\subsection{Wirtinger number}
We emphasize again that for a spatial graph diagram, there are various reasonable ways to define Wirtinger number. We choose to discuss versions that are compatible with the ascending number. We first define it so that the general bound in Theorem \ref{thm:mainbridgebound} works. Then, we will also modify the definition slightly, so that we get sharp bounds for $\theta$-curves. Again, whenever we would like to restrict the types of trees, we will decorate $\omega(G)$ with a subscript.

\begin{definition}[Wirtinger number of a virtual spatial graph diagram]
Let \(D\) be an oriented virtual spatial graph diagram. We ignore the virtual
crossings and impose Wirtinger relations only at the classical crossings.
Let \(\mathcal A(D)\) denote the set of strands of \(D\), where a strand is a
component of the diagram obtained by cutting \(D\) at each classical
undercrossing. \smallskip

A subset \(S\subset \mathcal A(D)\) is called a \emph{Wirtinger generating
set} for \(D\) if every strand of \(D\) can be obtained from \(S\) by repeatedly
applying the following rule. At a classical crossing, suppose that \(x_o\)
is the over-arc and \(x_i,x_j\) are the two under-arcs. If \(x_o\) and one
of \(x_i,x_j\) have already been obtained, then the remaining under-strand may
also be obtained, using the Wirtinger relation at the crossing.\smallskip

Equivalently, after choosing the usual convention for positive and negative
crossings, the relation at a crossing has the form
\[
        x_j=x_o^{\epsilon}x_i x_o^{-\epsilon},
        \qquad \epsilon\in\{+1,-1\}.
\]

Thus, knowing \(x_o\) and \(x_i\) determines \(x_j\), and knowing \(x_o\)
and \(x_j\) determines \(x_i\).\smallskip

The \emph{Wirtinger number} of the diagram \(D\) is
\[
        \omega(D)
        =
        \min\{|S|:S\subset \mathcal A(D)
        \text{ is a Wirtinger generating set for }D\}.
\]
\end{definition}

\begin{definition}
Let $G$ be a virtual (resp. welded) spatial graph. The \emph{virtual Wirtinger number} (resp. \emph{welded Wirtinger number}) of $G$ is
\[
        \omega(G)
        =
        \min\{\omega(D):D\text{ is a virtual spatial graph diagram for }G
        \text{ up to virtual (resp. welded) equivalence}\}.
\]

\end{definition}
\begin{remark}
\label{rem:coloring}
    Obtaining a new strand from a Wirtinger generating set can be visualized as extending the coloring from an original collection of colored strands. The extension rule is as follows: if the overstrand is colored and the incoming strand is colored, extend the coloring to the outgoing strand. The choice of actual colors used (red, green, blue etc.) does not matter for the computations of Wirtinger numbers. However, in some cases, it's convenient to impose some coloring rules. For instance, in Figure \ref{fig:tri}, we have used the rule that at each crossing, three distinct colors appear in order to exhibit a homomorphism to the symmetric group $S_3$.
\end{remark}
\begin{example}
    Consider the spatial graph $G$ depicted as a diagram in Figure \ref{fig:tri}. The strands colored red and blue form a Wirtinger generating set. In particular, the gray ovals on the left image are the places where the overstrands are colored and incoming strands are colored. The colorings are then extended in the right image. Thus, $\omega(G)\leq 2.$ We will see in the subsection on homomorphisms of spatial graph groups that this bound is actually sharp.
\end{example}

\begin{figure}[!ht]
    \centering\includegraphics[width=.7\textwidth]{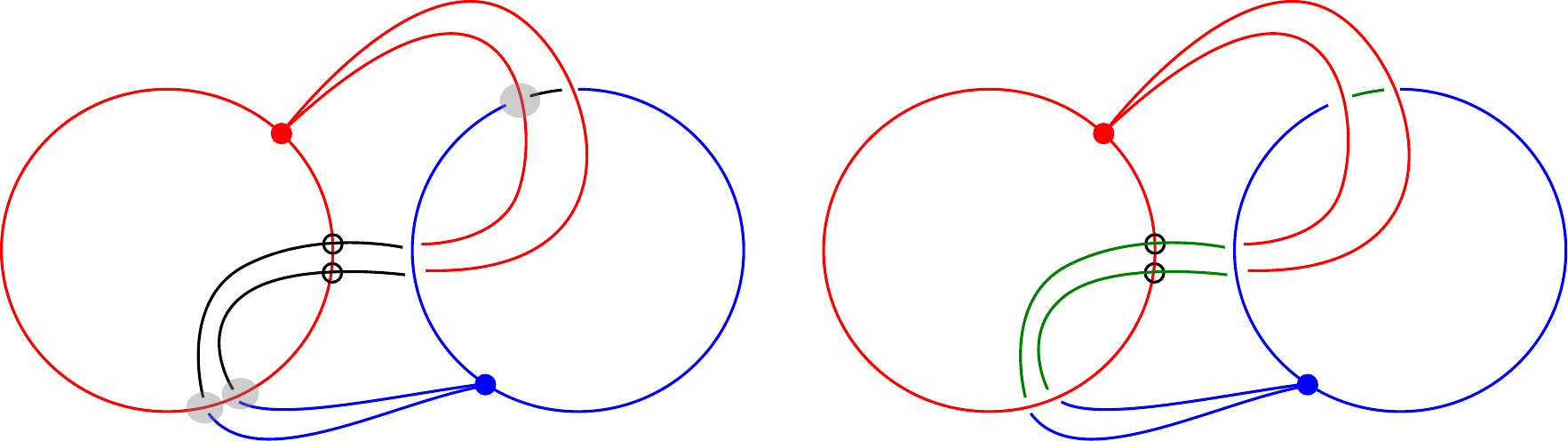}
    \caption{The Wirtinger number of this virtual spatial graph is 2.}
    \label{fig:tri}
\end{figure}

For $\theta$-curves, we will be modifying the definition of strands a bit to give a stronger bound for the ascending number. We will use the notation $\omega_{\theta}(G)$ to mean the Wirtinger number defined in the exact same way as in the beginning of the section except a strand includes only an arc between undercrossings that may include a bivalent vertex. This is better illustrated in Figure \ref{fig:kinowirtseq}, where the initial blue strand goes over a vertex of the $\theta$-graph. This is more restrictive than Figure \ref{fig:tri}, where a strand can include all arcs adjacent to a 4-valent vertex.

 \subsection{Spatial graph groups and quandles}
 A lower bound for the Wirtinger number of knotted objects is via the existence of surjective homomorphisms from the knotted object group to other well-understood groups. We will review how this works after reminding the readers of spatial graph groups. \smallskip

 Given a virtual spatial graph diagram, one can associate a fundamental group that does not change under virtual and welded Reidemeister moves as follows. Label each strand with a letter $\{x_1,x_2,\cdots,x_n\}$. At each crossing, depending on the sign of positive or negative crossing, the relation at a crossing has the form
\[
        x_j=x_o^{\epsilon}x_i x_o^{-\epsilon},
        \qquad \epsilon\in\{+1,-1\}.
\]
If a strand contains a vertex, all edges connected to these vertices receive the same letter.\smallskip

 To get a lower bound for the Wirtinger number, take a diagram realizing the Wirtinger number. Label the strands in the Wirtinger generating set as $S=\{x_1,x_2,\cdots,x_{\omega(G)}\}$. The fact that we have a Wirtinger generating set means that any other strand can be labeled as words in the letters from $\{x_1,x_2,\cdots,x_{\omega(G)}\}$. In other words, $\{x_1,x_2,\cdots,x_{\omega(G)}\}$ generates the spatial graph group of $G$. Now, define a mapping $f$ from $S$ to a symmetric group $S_n$ sending each $x_i$ to a transposition so that these transpositions generate $S_n$. We also would like the relations at each crossing to still hold using the images under $f.$ Since the symmetric group $S_n$ cannot be generated by fewer than $n-1$ transpositions, the Wirtinger number is at least $n-1$.

\begin{remark}\label{rem:counting}
    The preceding method relies on the existence of a single homomorphism to a symmetric group, whose structure is relatively well understood. For other algebraic structures, such as quandles, identifying a suitable homomorphism may be more difficult, and it can instead be useful to count all homomorphisms to a fixed finite target. Proposition~4.5 of \cite{blackwell2026bridge} gives a lower bound for the bridge number in terms of the number of homomorphisms from the fundamental quandle to a finite quandle. We will use only the consequence that, for a fixed finite quandle, if the number of such homomorphisms tends to infinity, then the bridge number also tends to infinity.
\end{remark}
\begin{example}[Example 1 revisited]
  The coloring in Figure~\ref{fig:tri} uses at least two distinct colors and assigns three distinct colors at each crossing. Interpreting the colors as transpositions in $S_3$, this coloring therefore determines a surjective homomorphism to $S_3$, since any two distinct transpositions generate $S_3$. Consequently, the group cannot admit a Wirtinger generating set of size one. \smallskip

The same conclusion can be obtained from the counting perspective described in Remark~\ref{rem:counting}. The colorings in which three distinct colors appear at each crossing correspond to homomorphisms from the fundamental quandle to the dihedral quandle of order $3$. There are six such nonconstant colorings, obtained by permuting the three colors. Together with the three constant colorings, this gives nine quandle homomorphisms in total. Since a one-element Wirtinger generating set would yield at most three homomorphisms to a quandle of order $3$, the Wirtinger generating set must have size at least two.

\end{example}

\subsection{Relating bridge numbers with Wirtinger numbers}

The following observation is straightforward.

\begin{prop}
   Let $G$ be a virtual or welded spatial graph. Then, $\omega(G)\leq bridge(G)$
\end{prop}
\begin{proof}
  We show that the set $\mathcal A(D)$ of overpasses is a Wirtinger generating set. By Remark~\ref{rem:coloring}, it suffices to show that, after assigning colors to the strands in $\mathcal A(D)$, the coloring can be extended to every strand of $D$ using coloring extension moves.\smallskip

First, observe that the overstrand at every classical crossing belongs to an overpass and is therefore colored initially. Moreover, an initially uncolored strand contains no classical overcrossing, although it may contain virtual crossings. Since virtual crossings impose no coloring conditions, they play no role in the argument.\smallskip

We now work one component at a time. Let $K$ be a component of $D$. By hypothesis, $K$ contains at least one overpass. Consequently, at least one strand of $K$ is colored initially. This provides a starting point from which the coloring can be propagated to every other strand of $K$.\smallskip

Suppose that some strand of $K$ remains uncolored. Starting from an initially colored overpass, traverse $K$ in either direction until an uncolored strand is encountered for the first time. Let $s$ be the last colored strand encountered and let $s'$ be the first uncolored strand encountered. The strands $s$ and $s'$ meet as the two understrands at a classical crossing $x$, since passing through a virtual crossing does not change the strand. The overstrand at $x$ belongs to an overpass and is therefore colored initially. Thus, at $x$, the overstrand and the understrand $s$ are colored, so a coloring extension move assigns a color to $s'$.\smallskip

Repeating this procedure propagates the coloring along $K$ until every strand of $K$ is colored. The hypothesis that every component contains an overpass ensures that this propagation process has an initial colored strand on each component. Applying the argument separately to every component of $D$ therefore colors every strand of the diagram. Hence, $\mathcal A(D)$ is a Wirtinger generating set.
\end{proof}

The reverse inequality also holds when $G$ is a classical link \cite{BlairKjuchukovaVelazquezVillanueva2020}, virtual link/welded link \cite{Pongtanapaisan2019}, or classical spatial graph \cite{blackwell2026bridge}.

\section{Ascending number}\label{sec:ascend}

In this section, we define the ascending number of virtual and welded spatial graphs. For this section, let $G$ be a virtual, respectively welded, spatial graph with underlying abstract graph $\Gamma$. Recall that a \textit{maximal tree} (also called a \textit{spanning tree}) for a graph is a subgraph which is a tree and which connects every vertex of the original graph.\smallskip

A \emph{base} for $G$ is a maximal forest $T$ of $\Gamma$. If $\Gamma$ is connected, then $T$ is a maximal tree. Let $D$ be a virtual, respectively welded, diagram of $G$. The subdiagram of $D$ corresponding to $T$ will be denoted by $D_T$. We say that $(D,T)$ is a \emph{based diagram} if no classical crossing of $D$ has a branch lying on $D_T$. \smallskip

 We begin with a lemma that is needed in the definition.

\begin{lemma}
    There exists a based diagram $(D,T)$ representing $G$. That is, every classical crossing of $D$ occurs among the edge diagrams corresponding to the edges of $\Gamma \setminus T$.
\end{lemma}

\begin{proof}
    \textbf{Case 1 ($G$ is a virtual spatial graph)}: Since each component of $T$ is a tree, there is a path $P_x$ from each point $x$ in the interior of $T$ to a leaf. So, let $x$ be a classical crossing point on $T.$ We move $x$ along $P_x$ via Reidemeister moves toward a leaf as shown in Figure \ref{fig:dragall} until $T$ does not contain $x$. Additional crossings will be created elsewhere, but not on $T.$ We note that even though more virtual and classical crossings are not depicted in the bottom of Figure \ref{fig:dragall}, the readers should imagine they are there and use the first row of Figure \ref{fig:dragall} for the general case. In Figure \ref{fig:pastvertex}, we show that moving a classical crossing point across a vertex can be realized as a sequence of graph Reidemeister moves. \smallskip

    \textbf{Case 2 ($G$ is a welded spatial graph)}: The idea is very similar with Case 1, with two main differences. One difference is in the welded case, when a strand goes under a vertex, a virtual crossing arises (see moves (IV$'$) and (III$'$) in Figure \ref{fig:WeldReid}). Nevertheless, it is still possible to get rid of a crossing on $T$. 
    
    The second difference is move (I$'$) may be needed when dealing with a bivalent vertex. First, due to move (VII$'$), we can assume that there is at most one bivalent vertex on an edge because more vertices will cancel out. In addition, moves (IV$'$), (III$'$) and (I$'$) can be performed to remove bivalent vertices as demonstrated in Figure \ref{fig:contract}.
\end{proof}

\begin{figure}[!ht]
    \centering
    \includegraphics[width=1\textwidth]{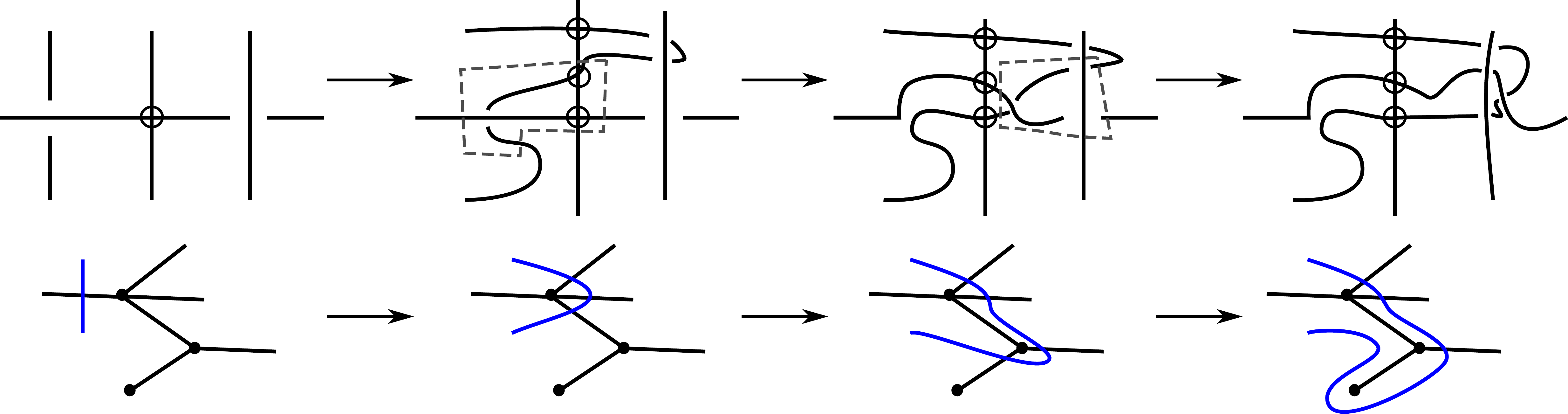}
    \caption{(Top) Dragging a classical crossing point past a virtual crossing and another classical crossing. (Bottom) Dragging a classical crossing point across a vertex and eventually over a leaf. The crossing is unspecified because we keep the same crossing type (over or under) throughout.}
    \label{fig:dragall}
\end{figure}

\begin{figure}[!ht]
    \centering
    \includegraphics[width=.7\textwidth]{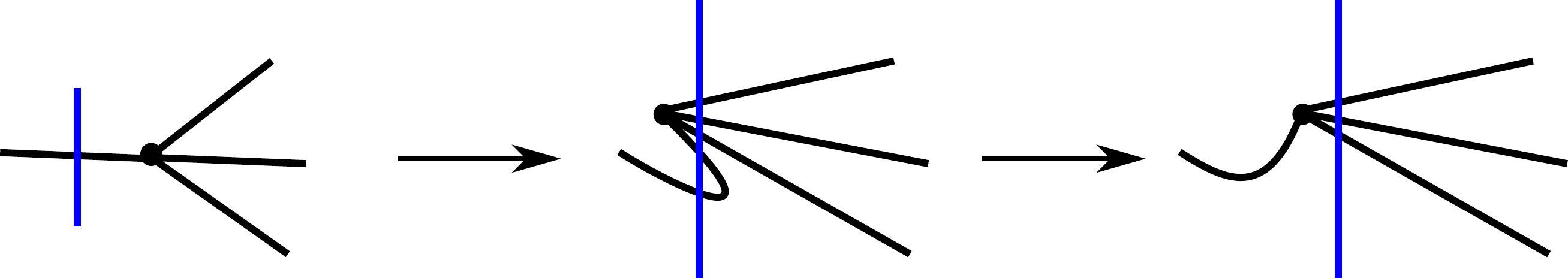}
    \caption{A sequence of Reidemeister moves.}
    \label{fig:pastvertex}
\end{figure}

\begin{figure}[!ht]
    \centering
    \includegraphics[width=.4\textwidth]{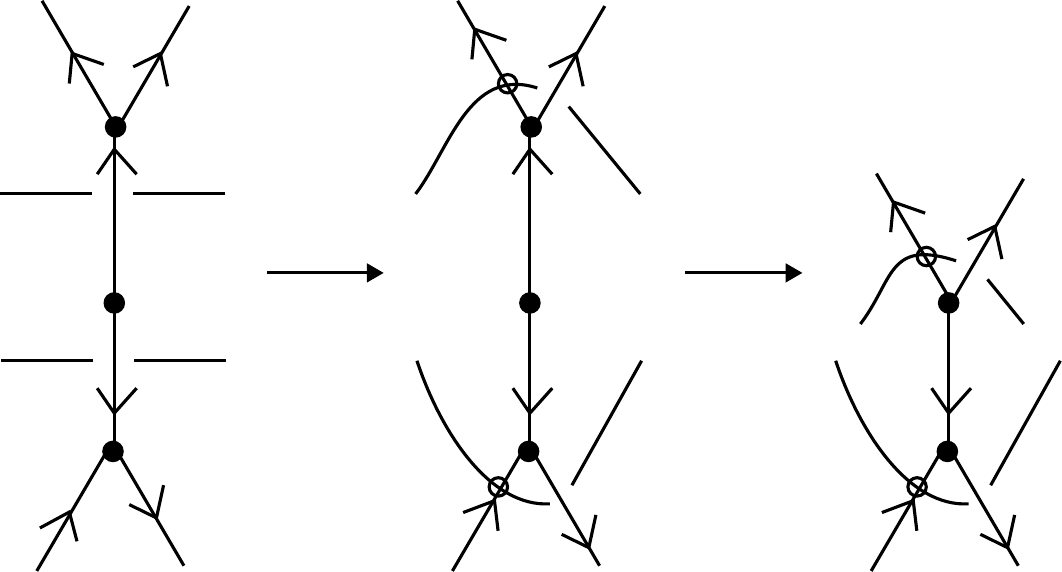}
    \caption{Move (I$'$) can be performed to remove bivalent vertices.}
    \label{fig:contract}
\end{figure}

Now, we are ready for the definition of ascending numbers.

\begin{definition}
    Let $e_1,e_2,\ldots,e_m$ be an ordering of the edges of $\Gamma \setminus T$, and choose an orientation on each $e_i$. We say that $(D,T)$ is \emph{descending with respect to this ordered oriented edge system} if the following conditions hold.

\begin{enumerate}
\item For each $i$, every classical self-crossing of the edge diagram $D_{e_i}$ is first encountered as an overcrossing when $D_{e_i}$ is traversed according to its chosen orientation.

\item If $i<j$, then at every classical crossing between $D_{e_i}$ and $D_{e_j}$, the edge diagram $D_{e_i}$ passes over $D_{e_j}$.

\end{enumerate}

Virtual crossings are ignored in these conditions. We say that $(D,T)$ is \emph{descending} if it is descending with respect to some ordering and orientation of the edges of $\Gamma \setminus T$. \smallskip

The \emph{ascending number} of a based diagram $(D,T)$, denoted by $a(D,T),$ is the minimum number of crossing changes at classical crossings required to transform $(D,T)$ into a descending based diagram. \smallskip

The \emph{ascending number} of a virtual, respectively welded, spatial graph $G$ is
\begin{center}
    $a(G) = \min \ \{ a(D,T) \mid (D,T) \text{ is a virtual, respectively welded, based diagram of } G \}.$
\end{center}
\end{definition}

 Equivalently, $a(G)$ is the minimum, over all diagrams $D$ of $G$ and all bases $T$, of the number of classical crossing changes needed to make the non-tree edges descending from first to last.

\section{Proofs of main results}
Arguably, the easiest way to unknot any classical knot $K$ is to take a diagram $D$ representing $K$, put a base point and an orientation on $K$, and perform crossing changes until one obtains a \textit{descending diagram}. To be descending means that as we start at the basepoint and start traversing the knot following the orientation, we encounter each crossing first as an overcrossing. This is due to the fact that one can obtain a one bridge presentation from a descending classical knot diagram and the only 1-bridge classical knot is the unknot.\smallskip

Ozawa in \cite{Ozawa2010} noticed that while this process accomplishes the task, it is not the most efficient way. That is, the minimum number of crossing changes needed to transform a knot diagram to ``any" diagram representing the unknot can be much lower than the minimum number of crossing changes needed to transform a knot diagram to a descending diagram. He used the bridge number to formally prove this intuitive fact. We will use the welded bridge number to obtain an analogous conclusion for welded knots. The following theorem is proven by Satoh in \cite{Satoh2018}.

\begin{theorem}[\cite{Satoh2018}]
A descending welded knot diagram represents the unknot. In particular, any welded knot diagram can be transformed into a diagram for the unknot by a sequence of extended Reidemeister moves, welded moves, planar isotopies and crossing changes.
\end{theorem}

We now generalize Ozawa's result to virtual spatial graphs. First, we make an observation that descending diagrams have minimal bridge numbers.

\mainbridgebound*

\begin{proof}
Let $D$ be a based diagram of a spatial graph $G$, and let $T$ be the maximal forest that determines the prescribed traversal of the edges of $D$. Suppose that $G$ has $n$ connected components. Then $T$ has exactly $n$ components.\smallskip

Let $u(D,T)$ denote the number of classical crossings that are first encountered along their understrands during the prescribed traversal. These are precisely the crossings that must be changed in order to make $D$ descending with respect to $T$.\smallskip

We construct a collection of overpasses in $D$. First, regard each component of $T$ as an initial overpass. This contributes $n$ overpasses. We then traverse the remaining edges in the prescribed order. If a crossing is encountered for the first time along its overstrand, then that overstrand can be included in the overpass currently being traversed, and hence the number of overpasses does not increase.\smallskip

On the other hand, suppose that a crossing $c$ is encountered for the first time along its understrand. The overstrand of $c$, when it is encountered later in the traversal, lies in an overpass that is not forced to be one of the initial overpasses arising from the components of $T$. We associate to $c$ the maximal overpass containing this overstrand. Thus, every overpass other than the initial $n$ overpasses is associated with at least one crossing that is first encountered as an undercrossing. Since several such crossings may determine the same overpass, the number of non-initial overpasses is at most $u(D,T)$.\smallskip

Consequently, $D$ admits a bridge decomposition with at most $n+u(D,T)$ overpasses. Therefore,
$$
b(G) \leq n+u(D,T).
$$
Taking the minimum over all based diagrams and all choices of maximal forest gives
$$
b(G) \leq n+a(G),
$$
and hence
$$
a(G) \geq b(G)-n.
$$

\end{proof}

We now give examples of $\theta$-curves, virtual, and welded knots whose unknotting number is low, but whose ascending number is high. This is the part where the Wirtinger number interpretation is useful.

\arbitrarilylarge*

\begin{proof}

\begin{enumerate}

\item[(a)]
Consider the Kinoshita $\theta$-curve $K$. The top of
Figure~\ref{fig:vertexsum} shows two copies of $K$, while the bottom
shows their vertex sum, which we denote by $K\#_3 K$. Before forming
the vertex sum, we perform a Reidemeister move~(VI) on the second copy
of $K$. This twist is introduced so that the resulting diagram is
compatible with our computation of the ascending number.

Let
\[
G=\#_3^n K
\]
denote the $n$-fold vertex sum of $K$, where a Reidemeister move~(VI)
is performed at each vertex sum. We claim that
\[
u(G)\leq n
\qquad\text{and}\qquad
a(G)=2n.
\]

\medskip

\noindent
\textbf{Claim.}
\[
u(G)\leq n.
\]

Figure~\ref{fig:unknotkino} shows that changing the highlighted
crossing transforms $K$ into the unknotted $\theta$-curve while
leaving the right-hand vertex fixed. We first perform this crossing
change in the leftmost summand of $G$. Since the right-hand vertex
remains fixed, the resulting unknotted summand can be isotoped into a
planar position without affecting the remaining summands.

We then repeat this procedure from left to right. One crossing change
is sufficient to unknot each summand. Therefore, $n$ crossing changes
transform $G$ into the unknotted graph of the same abstract type, and
hence
\[
u(G)\leq n.
\]

\medskip

\noindent
\textbf{Claim.}
\[
a(G)=2n.
\]

We first prove the upper bound. Figure~\ref{fig:vertexsum} shows that
changing two crossings in each summand produces a descending diagram.
The Reidemeister move~(VI) twists are chosen so that these local
crossing changes are compatible across the vertex sums and produce a
descending diagram of the entire graph. Thus,
\[
a(G)\leq 2n.
\]

For the lower bound, we use the bridge-number calculation from
\cite{blackwell2026bridge}. It is shown there that a minimum
unweighted bridge presentation of the $n$-fold vertex sum of the
Kinoshita $\theta$-curve has $n+1$ connected components. Although the
diagrams considered in \cite{blackwell2026bridge} do not contain the
Reidemeister move~(VI) twists used here, these local twists do not
affect the bridge-number calculation.

To compare the bridge number with the ascending number, we require the
two trivalent vertices to lie in the same ball. Write the bridge
decomposition as
\[
(S^3,G)
=
(B^3_+,T_+)\cup_S(B^3_-,T_-),
\]
where both trivalent vertices lie in $B^3_+$. By the unweighted
bridge-number calculation, $T_+$ has at least $n+1$ connected
components.

The components of $T_+$ consist of trivial trees and arcs. Since
$T_+$ contains both trivalent vertices, these $n+1$ components
contribute at least
\[
2n+4
\]
intersection points with the bridge sphere $S$. Every component of
$T_-$ is an arc with two endpoints on $S$. Therefore, $T_-$ contains
at least
\[
\frac{2n+4}{2}=n+2
\]
arcs.

The bridge number used in Theorem~\ref{thm:mainbridgebound} is
determined by the number of arcs in $T_-$. If $T_-$ contains $n+2$
arcs, then
\[
\operatorname{bridge}_{\theta}(G)
=
2(n+2)-3
=
2n+1.
\]
This value is realized by the bridge presentation obtained from the
construction in \cite{blackwell2026bridge}. Hence,
\[
\operatorname{bridge}_{\theta}(G)=2n+1.
\]

In particular, for one summand the bridge number is $3$, for two
summands the bridge number is $5$, and, in general, for $n$ summands
the bridge number is $2n+1$.

Applying Theorem~\ref{thm:mainbridgebound}, we obtain
\[
a(G)
\geq
\operatorname{bridge}_{\theta}(G)-1
=
(2n+1)-1
=
2n.
\]
Combining this with the upper bound gives
\[
a(G)=2n.
\]

Since $u(G)\leq n$, it follows that
\[
a(G)-u(G)
\geq
2n-n
=
n.
\]
Therefore, the difference between the ascending number and the
unknotting number can be made arbitrarily large. \qed

\item[(b)]
For each $n$, let $K_n$ be the knot constructed by Baader and
Kjuchukova in \cite{BaaderKjuchukova2020}. These knots satisfy
\[
u(K_n)=1
\qquad\text{and}\qquad
b(K_n)=n.
\]
Each $K_n$ is a satellite knot with Whitehead pattern. Moreover,
$K_n$ admits a surjective homomorphism
\[
\varphi_n\colon
\pi_1(S^3\setminus K_n)
\longrightarrow
S_n
\]
that sends meridians to transpositions.

Let $H_n$ be the handcuff graph obtained by replacing two parallel
strands of $K_n$ by the local configuration shown in
Figure~\ref{fig:whiteheadhandcuff}. The figure illustrates the
construction when $n=5$.

\medskip

\noindent
\textbf{Claim.}
\[
u(H_n)\leq 1.
\]

As shown in Figure~\ref{fig:whiteheadhandcuff}, changing the
highlighted crossing removes the Whitehead clasp. After this crossing
change, the two loop edges and the edge joining them can be isotoped
into a plane. The resulting spatial graph is therefore the unknotted
handcuff graph. Hence,
\[
u(H_n)\leq 1.
\]

\medskip

\noindent
\textbf{Claim.}
\[
\lim_{n\to\infty}
\bigl(a(H_n)-u(H_n)\bigr)
=
\infty.
\]

Recall that an associated link of a handcuff graph is obtained by
replacing a regular neighborhood of the non-loop edge with a band.
Any number of full twists may be inserted into this band. The
collection of links obtained in this way is an invariant of the
spatial handcuff graph.

Let $L_{n,m}$ denote the associated link obtained from $H_n$ by
inserting $m$ full twists into the band. The local replacement used
to construct $H_n$ does not affect the symmetric-group quotient
inherited from $K_n$. Moreover, the coloring extends across the band,
regardless of the number of full twists. Consequently, for every
integer $m$, there is a surjective homomorphism
\[
\psi_{n,m}\colon
\pi_1(S^3\setminus L_{n,m})
\longrightarrow
S_n
\]
that sends meridians to transpositions.

Any collection of transpositions generating $S_n$ must contain at
least $n-1$ elements. Indeed, associate a graph to such a collection
by taking its vertex set to be
\[
\{1,\ldots,n\}
\]
and adding an edge between $i$ and $j$ for each transposition $(ij)$
in the collection. If the transpositions generate $S_n$, then this
graph must be connected. Since a connected graph on $n$ vertices has
at least $n-1$ edges, at least $n-1$ transpositions are required.

It follows that
\[
\operatorname{mr}(L_{n,m})\geq n-1,
\]
where $\operatorname{mr}$ denotes meridional rank. Since meridional
rank is bounded above by bridge number, we have
\[
b(L_{n,m})\geq n-1.
\]

Thus, the bridge numbers of all associated links of $H_n$ tend to
infinity as $n$ tends to infinity. The bridge-number bound coming
from associated links therefore implies that
\[
b(H_n)\longrightarrow\infty
\qquad\text{as}\qquad
n\longrightarrow\infty.
\]
By Theorem~\ref{thm:mainbridgebound}, it follows that
\[
a(H_n)\longrightarrow\infty
\qquad\text{as}\qquad
n\longrightarrow\infty.
\]
Since $u(H_n)\leq 1$ for every $n$, we conclude that
\[
\lim_{n\to\infty}
\bigl(a(H_n)-u(H_n)\bigr)
=
\infty.
\]

\end{enumerate}

\end{proof}

\begin{figure}[ht!]
    \centering
    \includegraphics[width=.5\textwidth]{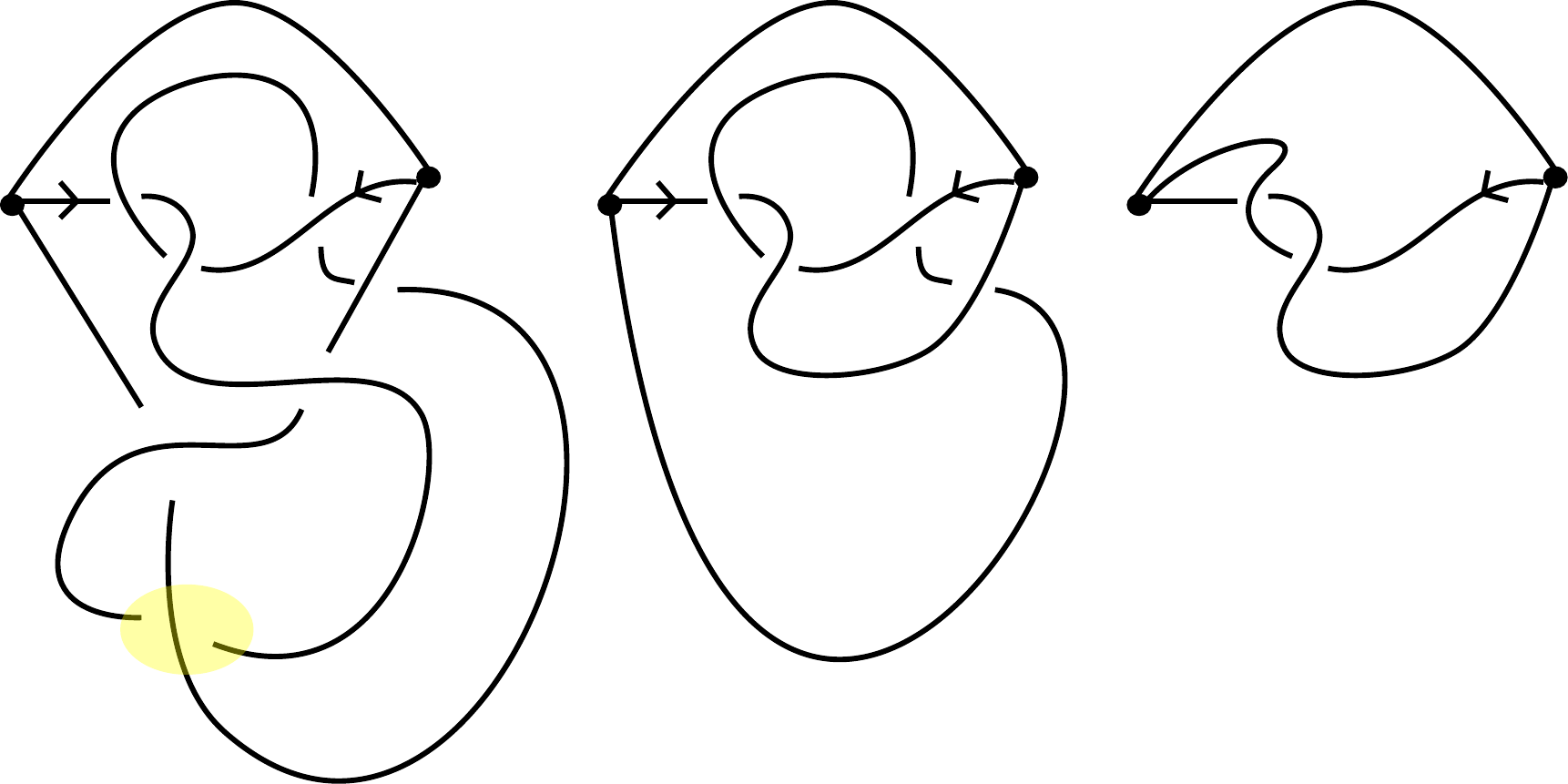}
    \caption{Performing one crossing change will turn the Kinoshita graph into the trivial $\theta$-curve.}
    \label{fig:unknotkino}
\end{figure}

\begin{figure}[ht!]
    \centering
    \includegraphics[width=.35\textwidth]{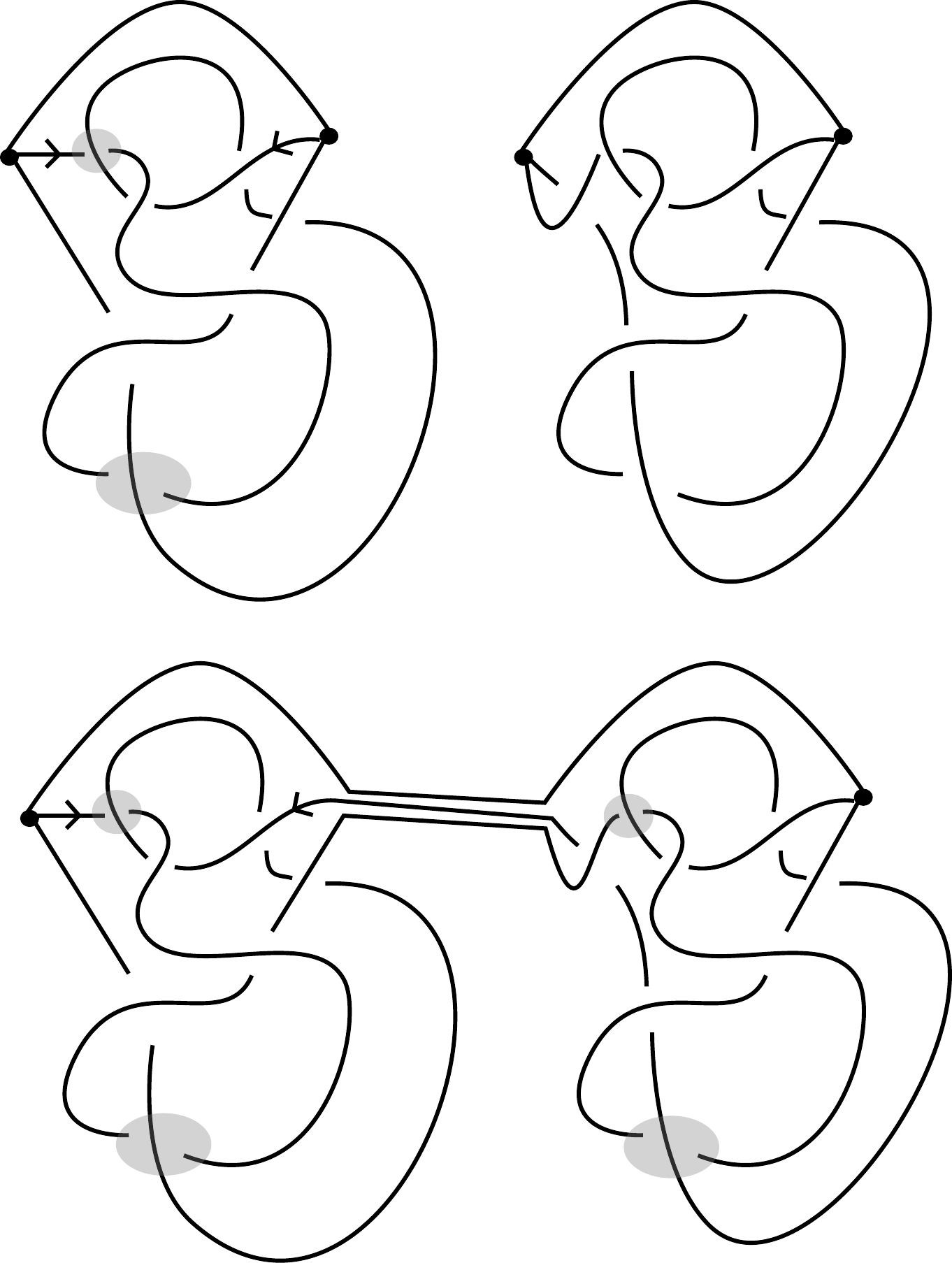}
    \caption{(Top) Two copies of the Kinoshita $\theta$-curve. The highlighted crossings are the two crossings whose crossing yields a descending diagram for the Kinoshita $\theta$-curve. (Bottom) A vertex sum of Kinoshita $\theta$-graphs that makes ascending number additive.}
    \label{fig:vertexsum}
\end{figure}

\begin{figure}[ht!]
    \centering
    \includegraphics[width=.6\textwidth]{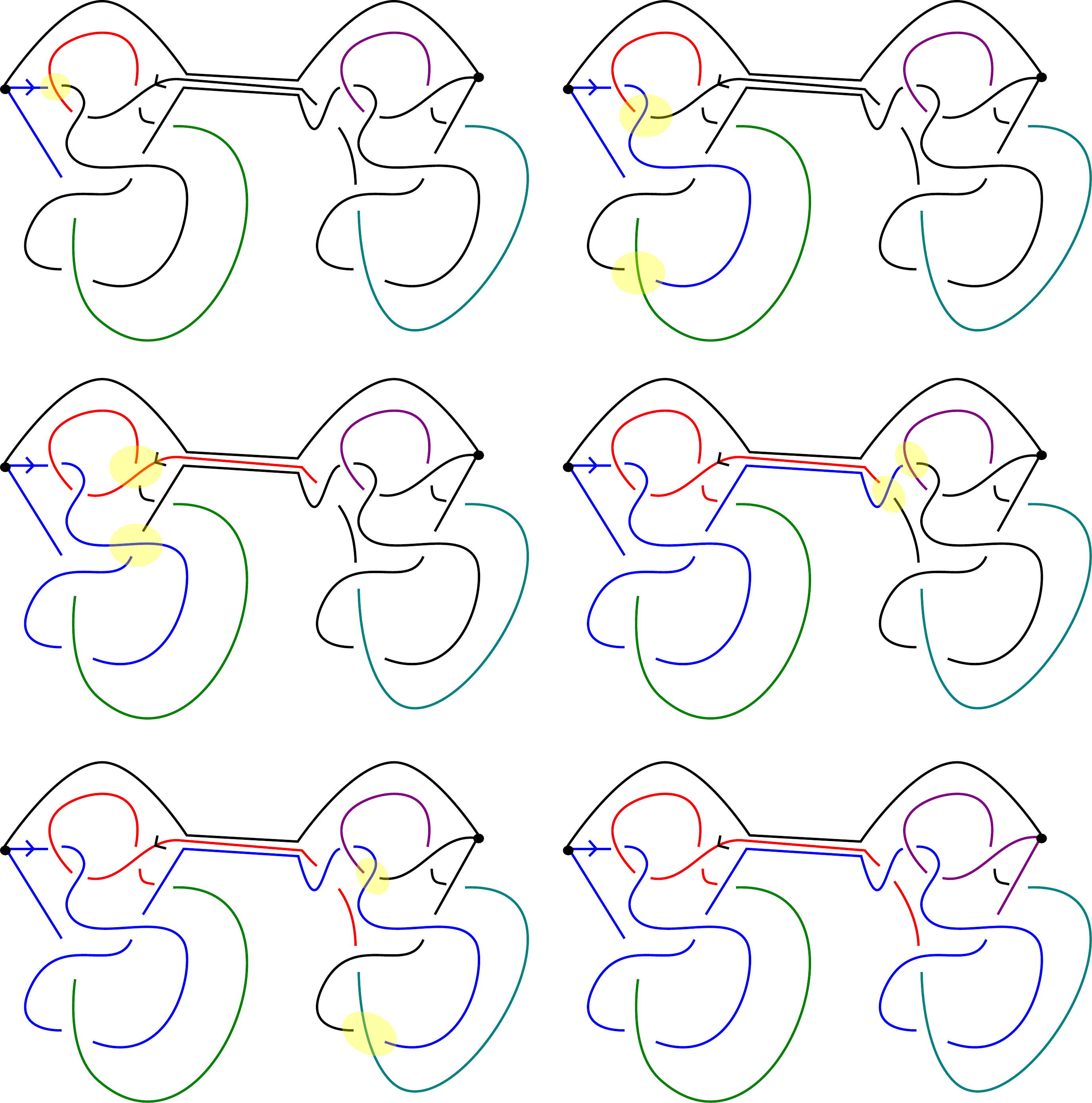}
    \caption{A sequence of coloring extension moves showing that $\omega_{\theta}\leq 5.$}
    \label{fig:kinowirtseq}
\end{figure}

\begin{figure}[H]
    \centering
    \labellist
\small\hair 2pt
\pinlabel $(34)$  at 1119 919
\pinlabel $(34)$  at 849 1000
\pinlabel $(35)$  at 965 981
\pinlabel $(45)$  at 1003 1119
\pinlabel $(45)$  at 933 1119
\pinlabel $(35)$  at 1039 799
\pinlabel $(35)$  at 890 799
\pinlabel $(34)$  at 960 699
\pinlabel $(45)$  at 1129 570
\pinlabel $(45)$  at 800 599
\pinlabel $(34)$  at 882 499
\pinlabel $(34)$  at 976 479

\pinlabel $(34)$  at 269 1069
\pinlabel $(45)$  at 361 1069
\endlabellist

    \includegraphics[width=.6\textwidth]{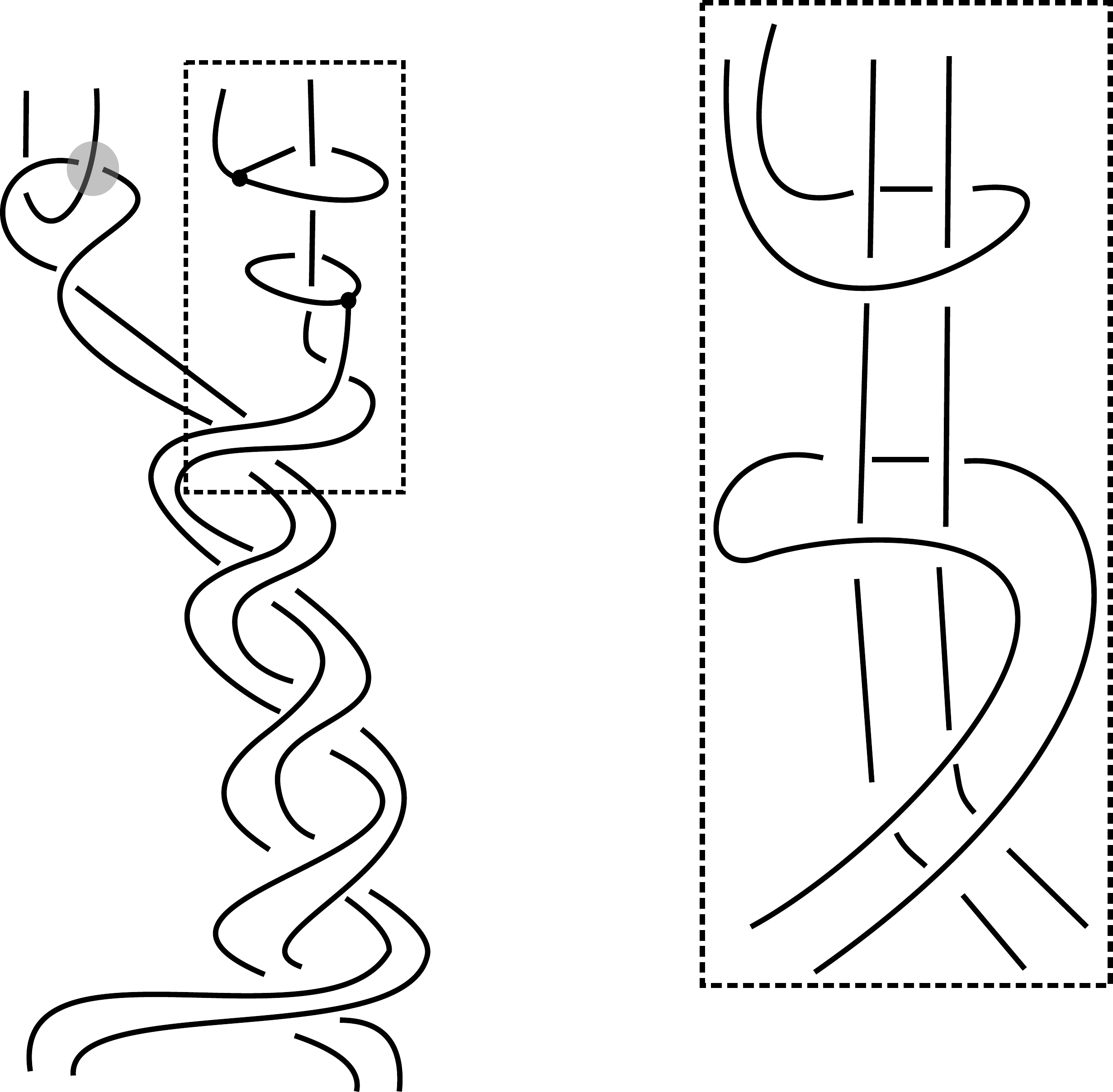}
    \caption{The reader should imagine the standard braid closure. After turning an edge into a band, we still have a surjective homomorphism to $S_n.$}
    \label{fig:whiteheadhandcuff}
\end{figure}

\section{Appendix: Computational work}
\label{sec:computational}
In this section, we begin to investigate ascending number and unknotting number of virtual knotted objects, focusing mainly on virtual knots. The computer code written by the third author is available here: https://github.com/ppongtan/VirtualGordian.

\subsection{A lower bound of unknotting number from slice genus}

It is expected that the slice genus will bound the number of crossing changes to transform to the unknot. The authors could not find it in other articles, so we produce a short proof here.

\begin{corollary}
    Suppose that virtual knots $K$ and $J$ differ by $n$ crossing changes. Then, there is a genus $n$ cobordism connecting $K$ and $J$.
\end{corollary}
\begin{proof}
    Each oriented band move represents a saddle. One can achieve a crossing change by first performing an oriented band move from $+1$ tangle to a crossing-less tangle, and then performing another oriented band move from the crossing-less tangle to the $-1$ tangle (See Figure \ref{fig:spins}).
\end{proof}

\begin{figure}[h]
\centering
\includegraphics[scale=0.4]{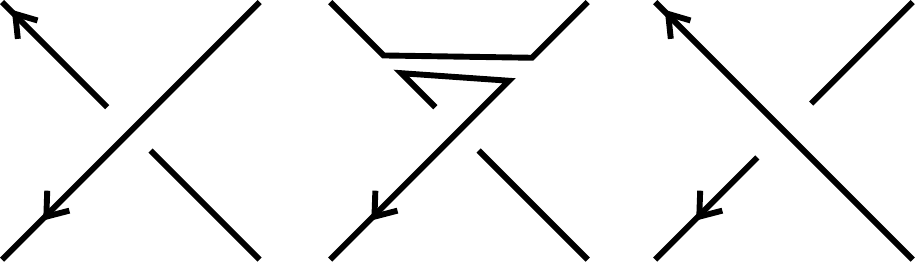}
\caption{One crossing change achieved by two oriented band surgeries.}
\label{fig:spins}
\end{figure}
Chrisman, Boden, and Gaudreau study virtual slice genus in \cite{boden2019virtual}. Their table can be used to help with lower estimates for the crossing change distance.
\begin{example}
    The virtual knot 4.53 can be turned into the unknot by 2 crossing changes. Furthermore, the slice genus is 2. Therefore, the unknotting number of 4.53 is precisely 2.
\end{example}

\subsection{Table of ascending number and unknotting number estimates}

We recall that Theorem \ref{thm:theta} provides useful context for interpreting the comparison between ascending and unknotting numbers. Since virtual crossings are not recorded in a Gauss code, readers are encouraged to reconstruct the corresponding virtual knot diagrams in order to visualize this comparison. We emphasize that the ascending number listed in the table is the ascending number of the particular diagram represented by the Gauss code and therefore provides only an upper bound for the ascending number of the underlying virtual knot type. By contrast, an entry in the $g_4$ or $u$ column gives the exact value for the virtual knot type unless it is preceded by $\leq$ or is given as an interval, such as $[1,2]$.

{\scriptsize
\begin{longtable}{|c|p{0.43\textwidth}|p{0.21\textwidth}|c|c|c|}

\hline
Virtual knot & Original Gauss code & Reduced Gauss code & $a(D)$ & $g_4$ & $u$ \\
\hline
\endfirsthead

\hline
Virtual knot & Original Gauss code & Reduced Gauss code & $a$& $g_4$ & $u$ \\
\hline
\endhead

$2.1$ & $\mathrm{O}1^{-}\ \mathrm{O}2^{-}\ \mathrm{U}1^{-}\ \mathrm{U}2^{-}$ & $\mathrm{O}2^{-}\ \mathrm{U}2^{-}$ & $0$ & 1 & $1$ \\ \hline
$4.3$ & $\mathrm{O}1^{-}\ \mathrm{O}2^{-}\ \mathrm{U}1^{-}\ \mathrm{U}2^{-}\ \mathrm{O}3^{-}\ \mathrm{U}4^{-}\ \mathrm{U}3^{-}\ \mathrm{O}4^{-}$ & $\mathrm{O}2^{-}\ \mathrm{U}2^{-}\ \mathrm{U}4^{-}\ \mathrm{O}4^{-}$ & $0$ & $2$ & $2$ \\ \hline
$4.6$ & $\mathrm{O}1^{-}\ \mathrm{O}2^{-}\ \mathrm{U}1^{-}\ \mathrm{U}2^{-}\ \mathrm{O}3^{+}\ \mathrm{U}4^{+}\ \mathrm{U}3^{+}\ \mathrm{O}4^{+}$ & $\mathrm{O}2^{-}\ \mathrm{U}2^{-}\ \mathrm{U}4^{+}\ \mathrm{O}4^{+}$ & $0$ & $1$ & $\leq 2$ \\ \hline
$4.53$ & $\mathrm{O}1^{-}\ \mathrm{O}2^{-}\ \mathrm{U}1^{-}\ \mathrm{O}3^{-}\ \mathrm{U}4^{-}\ \mathrm{U}3^{-}\ \mathrm{O}4^{-}\ \mathrm{U}2^{-}$ & $\mathrm{U}4^{-}\ \mathrm{O}4^{-}$ & $1$ & $2$ & $2$ \\ \hline
$4.54$ & $\mathrm{O}1^{-}\ \mathrm{O}2^{-}\ \mathrm{U}1^{-}\ \mathrm{O}3^{-}\ \mathrm{U}4^{+}\ \mathrm{U}3^{-}\ \mathrm{O}4^{+}\ \mathrm{U}2^{-}$ & $\mathrm{U}4^{+}\ \mathrm{O}4^{+}$ & $1$ & $1$ & $1$ \\ \hline
$4.73$ & $\mathrm{O}1^{-}\ \mathrm{O}2^{-}\ \mathrm{U}1^{-}\ \mathrm{U}3^{-}\ \mathrm{O}4^{-}\ \mathrm{O}3^{-}\ \mathrm{U}4^{-}\ \mathrm{U}2^{-}$ & $\mathrm{O}4^{-}\ \mathrm{U}4^{-}$ & $1$ & $2$ & $2$ \\ \hline
$4.74$ & $\mathrm{O}1^{-}\ \mathrm{O}2^{-}\ \mathrm{U}1^{-}\ \mathrm{U}3^{-}\ \mathrm{O}4^{+}\ \mathrm{O}3^{-}\ \mathrm{U}4^{+}\ \mathrm{U}2^{-}$ & $\mathrm{O}4^{+}\ \mathrm{U}4^{+}$ & $1$ & $1$ & $1$ \\ \hline
$4.75$ & $\mathrm{O}1^{-}\ \mathrm{O}2^{-}\ \mathrm{U}1^{-}\ \mathrm{U}3^{+}\ \mathrm{O}4^{+}\ \mathrm{O}3^{+}\ \mathrm{U}4^{+}\ \mathrm{U}2^{-}$ & $\mathrm{O}4^{+}\ \mathrm{U}4^{+}$ & $1$ & $0$ & $\leq 2$ \\ \hline
$5.37$ & $\mathrm{O}1^{-}\ \mathrm{O}2^{-}\ \mathrm{U}1^{-}\ \mathrm{U}2^{-}\ \mathrm{O}3^{-}\ \mathrm{O}4^{-}\ \mathrm{U}5^{-}\ \mathrm{U}4^{-}\ \mathrm{O}5^{-}\ \mathrm{U}3^{-}$ & $\mathrm{O}2^{-}\ \mathrm{U}2^{-}\ \mathrm{U}5^{-}\ \mathrm{O}5^{-}$ & $1$ & $2$ & $2$ \\ \hline
$5.39$ & $\mathrm{O}1^{-}\ \mathrm{O}2^{-}\ \mathrm{U}1^{-}\ \mathrm{U}2^{-}\ \mathrm{O}3^{-}\ \mathrm{O}4^{+}\ \mathrm{U}5^{+}\ \mathrm{U}4^{+}\ \mathrm{O}5^{+}\ \mathrm{U}3^{-}$ & $\mathrm{O}2^{-}\ \mathrm{U}2^{-}\ \mathrm{U}5^{+}\ \mathrm{O}5^{+}$ & $1$ & $1$ & $\leq 2$ \\ \hline
$5.40$ & $\mathrm{O}1^{-}\ \mathrm{O}2^{-}\ \mathrm{U}1^{-}\ \mathrm{U}2^{-}\ \mathrm{O}3^{+}\ \mathrm{O}4^{-}\ \mathrm{U}5^{-}\ \mathrm{U}4^{-}\ \mathrm{O}5^{-}\ \mathrm{U}3^{+}$ & $\mathrm{O}2^{-}\ \mathrm{U}2^{-}\ \mathrm{U}5^{-}\ \mathrm{O}5^{-}$ & $1$ & $[1,2]$ & $\leq 2$ \\ \hline
$5.42$ & $\mathrm{O}1^{-}\ \mathrm{O}2^{-}\ \mathrm{U}1^{-}\ \mathrm{U}2^{-}\ \mathrm{O}3^{+}\ \mathrm{O}4^{+}\ \mathrm{U}5^{+}\ \mathrm{U}4^{+}\ \mathrm{O}5^{+}\ \mathrm{U}3^{+}$ & $\mathrm{O}2^{-}\ \mathrm{U}2^{-}\ \mathrm{U}5^{+}\ \mathrm{O}5^{+}$ & $1$ & $1$ & $\leq 2$ \\ \hline
$5.104$ & $\mathrm{O}1^{-}\ \mathrm{O}2^{-}\ \mathrm{U}1^{-}\ \mathrm{U}2^{-}\ \mathrm{U}3^{-}\ \mathrm{O}4^{-}\ \mathrm{U}5^{-}\ \mathrm{U}4^{-}\ \mathrm{O}5^{-}\ \mathrm{O}3^{-}$ & $\mathrm{O}2^{-}\ \mathrm{U}2^{-}\ \mathrm{U}5^{-}\ \mathrm{O}5^{-}$ & $0$ & $2$ & $2$ \\ \hline
$5.106$ & $\mathrm{O}1^{-}\ \mathrm{O}2^{-}\ \mathrm{U}1^{-}\ \mathrm{U}2^{-}\ \mathrm{U}3^{-}\ \mathrm{O}4^{+}\ \mathrm{U}5^{+}\ \mathrm{U}4^{+}\ \mathrm{O}5^{+}\ \mathrm{O}3^{-}$ & $\mathrm{O}2^{-}\ \mathrm{U}2^{-}\ \mathrm{U}5^{+}\ \mathrm{O}5^{+}$ & $0$ & $1$ & $\leq 2$ \\ \hline
$5.107$ & $\mathrm{O}1^{-}\ \mathrm{O}2^{-}\ \mathrm{U}1^{-}\ \mathrm{U}2^{-}\ \mathrm{U}3^{+}\ \mathrm{O}4^{-}\ \mathrm{U}5^{-}\ \mathrm{U}4^{-}\ \mathrm{O}5^{-}\ \mathrm{O}3^{+}$ & $\mathrm{O}2^{-}\ \mathrm{U}2^{-}\ \mathrm{U}5^{-}\ \mathrm{O}5^{-}$ & $0$ & $[1,2]$ & $\leq 2$ \\ \hline
$5.109$ & $\mathrm{O}1^{-}\ \mathrm{O}2^{-}\ \mathrm{U}1^{-}\ \mathrm{U}2^{-}\ \mathrm{U}3^{+}\ \mathrm{O}4^{+}\ \mathrm{U}5^{+}\ \mathrm{U}4^{+}\ \mathrm{O}5^{+}\ \mathrm{O}3^{+}$ & $\mathrm{O}2^{-}\ \mathrm{U}2^{-}\ \mathrm{U}5^{+}\ \mathrm{O}5^{+}$ & $0$ & $1$ & $\leq 2$ \\ \hline
$5.877$ & $\mathrm{O}1^{-}\ \mathrm{O}2^{-}\ \mathrm{U}1^{-}\ \mathrm{O}3^{-}\ \mathrm{O}4^{-}\ \mathrm{U}5^{-}\ \mathrm{U}4^{-}\ \mathrm{O}5^{-}\ \mathrm{U}3^{-}\ \mathrm{U}2^{-}$ & $\mathrm{U}5^{-}\ \mathrm{O}5^{-}$ & $1$ & $2$ & $2$ \\ \hline
$5.878$ & $\mathrm{O}1^{-}\ \mathrm{O}2^{-}\ \mathrm{U}1^{-}\ \mathrm{O}3^{-}\ \mathrm{O}4^{+}\ \mathrm{U}5^{-}\ \mathrm{U}4^{+}\ \mathrm{O}5^{-}\ \mathrm{U}3^{-}\ \mathrm{U}2^{-}$ & $\mathrm{U}5^{-}\ \mathrm{O}5^{-}$ & $1$ & $1$ & $1$ \\ \hline
$5.879$ & $\mathrm{O}1^{-}\ \mathrm{O}2^{-}\ \mathrm{U}1^{-}\ \mathrm{O}3^{-}\ \mathrm{O}4^{+}\ \mathrm{U}5^{+}\ \mathrm{U}4^{+}\ \mathrm{O}5^{+}\ \mathrm{U}3^{-}\ \mathrm{U}2^{-}$ & $\mathrm{U}5^{+}\ \mathrm{O}5^{+}$ & $1$ & $1$ & $\leq 2$ \\ \hline
$5.880$ & $\mathrm{O}1^{-}\ \mathrm{O}2^{-}\ \mathrm{U}1^{-}\ \mathrm{O}3^{+}\ \mathrm{O}4^{-}\ \mathrm{U}5^{-}\ \mathrm{U}4^{-}\ \mathrm{O}5^{-}\ \mathrm{U}3^{+}\ \mathrm{U}2^{-}$ & $\mathrm{U}5^{-}\ \mathrm{O}5^{-}$ & $1$ & $[1,2]$ & $\leq 2$ \\ \hline
$5.881$ & $\mathrm{O}1^{-}\ \mathrm{O}2^{-}\ \mathrm{U}1^{-}\ \mathrm{O}3^{+}\ \mathrm{O}4^{-}\ \mathrm{U}5^{+}\ \mathrm{U}4^{-}\ \mathrm{O}5^{+}\ \mathrm{U}3^{+}\ \mathrm{U}2^{-}$ & $\mathrm{U}5^{+}\ \mathrm{O}5^{+}$ & $1$ & $1$ & $1$ \\ \hline
$5.1234$ & $\mathrm{O}1^{-}\ \mathrm{O}2^{-}\ \mathrm{U}1^{-}\ \mathrm{O}3^{-}\ \mathrm{U}4^{-}\ \mathrm{O}5^{-}\ \mathrm{O}4^{-}\ \mathrm{U}5^{-}\ \mathrm{U}3^{-}\ \mathrm{U}2^{-}$ & $\mathrm{O}5^{-}\ \mathrm{U}5^{-}$ & $1$ & $2$ & $2$ \\ \hline
$5.1235$ & $\mathrm{O}1^{-}\ \mathrm{O}2^{-}\ \mathrm{U}1^{-}\ \mathrm{O}3^{-}\ \mathrm{U}4^{-}\ \mathrm{O}5^{+}\ \mathrm{O}4^{-}\ \mathrm{U}5^{+}\ \mathrm{U}3^{-}\ \mathrm{U}2^{-}$ & $\mathrm{O}5^{+}\ \mathrm{U}5^{+}$ & $1$ & $1$ & $1$ \\ \hline
$5.1236$ & $\mathrm{O}1^{-}\ \mathrm{O}2^{-}\ \mathrm{U}1^{-}\ \mathrm{O}3^{-}\ \mathrm{U}4^{+}\ \mathrm{O}5^{-}\ \mathrm{O}4^{+}\ \mathrm{U}5^{-}\ \mathrm{U}3^{-}\ \mathrm{U}2^{-}$ & $\mathrm{O}5^{-}\ \mathrm{U}5^{-}$ & $1$ & $1$ & $1$ \\ \hline
$5.1237$ & $\mathrm{O}1^{-}\ \mathrm{O}2^{-}\ \mathrm{U}1^{-}\ \mathrm{O}3^{-}\ \mathrm{U}4^{+}\ \mathrm{O}5^{+}\ \mathrm{O}4^{+}\ \mathrm{U}5^{+}\ \mathrm{U}3^{-}\ \mathrm{U}2^{-}$ & $\mathrm{O}5^{+}\ \mathrm{U}5^{+}$ & $1$ & $1$ & $\leq 2$ \\ \hline
$5.1238$ & $\mathrm{O}1^{-}\ \mathrm{O}2^{-}\ \mathrm{U}1^{-}\ \mathrm{O}3^{+}\ \mathrm{U}4^{-}\ \mathrm{O}5^{-}\ \mathrm{O}4^{-}\ \mathrm{U}5^{-}\ \mathrm{U}3^{+}\ \mathrm{U}2^{-}$ & $\mathrm{O}5^{-}\ \mathrm{U}5^{-}$ & $1$ & $[1,2]$ & $\leq 2$ \\ \hline
$5.1239$ & $\mathrm{O}1^{-}\ \mathrm{O}2^{-}\ \mathrm{U}1^{-}\ \mathrm{O}3^{+}\ \mathrm{U}4^{-}\ \mathrm{O}5^{+}\ \mathrm{O}4^{-}\ \mathrm{U}5^{+}\ \mathrm{U}3^{+}\ \mathrm{U}2^{-}$ & $\mathrm{O}5^{+}\ \mathrm{U}5^{+}$ & $1$ & $1$ & $1$ \\ \hline
$5.1240$ & $\mathrm{O}1^{-}\ \mathrm{O}2^{-}\ \mathrm{U}1^{-}\ \mathrm{O}3^{+}\ \mathrm{U}4^{+}\ \mathrm{O}5^{-}\ \mathrm{O}4^{+}\ \mathrm{U}5^{-}\ \mathrm{U}3^{+}\ \mathrm{U}2^{-}$ & $\mathrm{O}5^{-}\ \mathrm{U}5^{-}$ & $1$ & $1$ & $1$ \\ \hline
$6.132$ & $\mathrm{O}1^{-}\ \mathrm{O}2^{-}\ \mathrm{U}1^{-}\ \mathrm{U}2^{-}\ \mathrm{O}3^{-}\ \mathrm{O}4^{+}\ \mathrm{U}3^{-}\ \mathrm{O}5^{-}\ \mathrm{U}6^{-}\ \mathrm{U}5^{-}\ \mathrm{O}6^{-}\ \mathrm{U}4^{+}$ & $\mathrm{O}2^{-}\ \mathrm{U}2^{-}\ \mathrm{U}6^{-}\ \mathrm{O}6^{-}$ & $1$ & $[1,2]$ & $\leq 2$ \\ \hline
$6.135$ & $\mathrm{O}1^{-}\ \mathrm{O}2^{-}\ \mathrm{U}1^{-}\ \mathrm{U}2^{-}\ \mathrm{O}3^{-}\ \mathrm{O}4^{+}\ \mathrm{U}3^{-}\ \mathrm{O}5^{+}\ \mathrm{U}6^{+}\ \mathrm{U}5^{+}\ \mathrm{O}6^{+}\ \mathrm{U}4^{+}$ & $\mathrm{O}2^{-}\ \mathrm{U}2^{-}\ \mathrm{U}6^{+}\ \mathrm{O}6^{+}$ & $1$ & $[1,2]$ & $\leq 2$ \\ \hline
$6.136$ & $\mathrm{O}1^{-}\ \mathrm{O}2^{-}\ \mathrm{U}1^{-}\ \mathrm{U}2^{-}\ \mathrm{O}3^{+}\ \mathrm{O}4^{-}\ \mathrm{U}3^{+}\ \mathrm{O}5^{-}\ \mathrm{U}6^{-}\ \mathrm{U}5^{-}\ \mathrm{O}6^{-}\ \mathrm{U}4^{-}$ & $\mathrm{O}2^{-}\ \mathrm{U}2^{-}\ \mathrm{U}6^{-}\ \mathrm{O}6^{-}$ & $1$ & $[1,2]$ & $\leq 2$ \\ \hline
$6.139$ & $\mathrm{O}1^{-}\ \mathrm{O}2^{-}\ \mathrm{U}1^{-}\ \mathrm{U}2^{-}\ \mathrm{O}3^{+}\ \mathrm{O}4^{-}\ \mathrm{U}3^{+}\ \mathrm{O}5^{+}\ \mathrm{U}6^{+}\ \mathrm{U}5^{+}\ \mathrm{O}6^{+}\ \mathrm{U}4^{-}$ & $\mathrm{O}2^{-}\ \mathrm{U}2^{-}\ \mathrm{U}6^{+}\ \mathrm{O}6^{+}$ & $1$ & $[1,2]$ & $\leq 2$ \\ \hline
$6.228$ & $\mathrm{O}1^{-}\ \mathrm{O}2^{-}\ \mathrm{U}1^{-}\ \mathrm{U}2^{-}\ \mathrm{O}3^{-}\ \mathrm{O}4^{+}\ \mathrm{U}3^{-}\ \mathrm{U}5^{-}\ \mathrm{O}6^{-}\ \mathrm{O}5^{-}\ \mathrm{U}6^{-}\ \mathrm{U}4^{+}$ & $\mathrm{O}2^{-}\ \mathrm{U}2^{-}\ \mathrm{O}6^{-}\ \mathrm{U}6^{-}$ & $1$ & $[1,2]$ & $\leq 2$ \\ \hline
$6.231$ & $\mathrm{O}1^{-}\ \mathrm{O}2^{-}\ \mathrm{U}1^{-}\ \mathrm{U}2^{-}\ \mathrm{O}3^{-}\ \mathrm{O}4^{+}\ \mathrm{U}3^{-}\ \mathrm{U}5^{+}\ \mathrm{O}6^{+}\ \mathrm{O}5^{+}\ \mathrm{U}6^{+}\ \mathrm{U}4^{+}$ & $\mathrm{O}2^{-}\ \mathrm{U}2^{-}\ \mathrm{O}6^{+}\ \mathrm{U}6^{+}$ & $1$ & $[1,2]$ & $\leq 2$ \\ \hline
$6.232$ & $\mathrm{O}1^{-}\ \mathrm{O}2^{-}\ \mathrm{U}1^{-}\ \mathrm{U}2^{-}\ \mathrm{O}3^{+}\ \mathrm{O}4^{-}\ \mathrm{U}3^{+}\ \mathrm{U}5^{-}\ \mathrm{O}6^{-}\ \mathrm{O}5^{-}\ \mathrm{U}6^{-}\ \mathrm{U}4^{-}$ & $\mathrm{O}2^{-}\ \mathrm{U}2^{-}\ \mathrm{O}6^{-}\ \mathrm{U}6^{-}$ & $1$ & $[1,2]$ & $\leq 2$ \\ \hline
$6.235$ & $\mathrm{O}1^{-}\ \mathrm{O}2^{-}\ \mathrm{U}1^{-}\ \mathrm{U}2^{-}\ \mathrm{O}3^{+}\ \mathrm{O}4^{-}\ \mathrm{U}3^{+}\ \mathrm{U}5^{+}\ \mathrm{O}6^{+}\ \mathrm{O}5^{+}\ \mathrm{U}6^{+}\ \mathrm{U}4^{-}$ & $\mathrm{O}2^{-}\ \mathrm{U}2^{-}\ \mathrm{O}6^{+}\ \mathrm{U}6^{+}$ & $1$ & $[1,2]$ & $\leq 2$ \\ \hline
$6.650$ & $\mathrm{O}1^{-}\ \mathrm{O}2^{-}\ \mathrm{U}1^{-}\ \mathrm{U}2^{-}\ \mathrm{O}3^{-}\ \mathrm{O}4^{-}\ \mathrm{O}5^{-}\ \mathrm{U}6^{-}\ \mathrm{U}5^{-}\ \mathrm{O}6^{-}\ \mathrm{U}4^{-}\ \mathrm{U}3^{-}$ & $\mathrm{O}2^{-}\ \mathrm{U}2^{-}\ \mathrm{U}6^{-}\ \mathrm{O}6^{-}$ & $1$ & $2$ & $2$ \\ \hline
$6.652$ & $\mathrm{O}1^{-}\ \mathrm{O}2^{-}\ \mathrm{U}1^{-}\ \mathrm{U}2^{-}\ \mathrm{O}3^{-}\ \mathrm{O}4^{-}\ \mathrm{O}5^{+}\ \mathrm{U}6^{+}\ \mathrm{U}5^{+}\ \mathrm{O}6^{+}\ \mathrm{U}4^{-}\ \mathrm{U}3^{-}$ & $\mathrm{O}2^{-}\ \mathrm{U}2^{-}\ \mathrm{U}6^{+}\ \mathrm{O}6^{+}$ & $1$ & $[1,2]$ & $\leq 2$ \\ \hline
$6.653$ & $\mathrm{O}1^{-}\ \mathrm{O}2^{-}\ \mathrm{U}1^{-}\ \mathrm{U}2^{-}\ \mathrm{O}3^{+}\ \mathrm{O}4^{+}\ \mathrm{O}5^{-}\ \mathrm{U}6^{-}\ \mathrm{U}5^{-}\ \mathrm{O}6^{-}\ \mathrm{U}4^{+}\ \mathrm{U}3^{+}$ & $\mathrm{O}2^{-}\ \mathrm{U}2^{-}\ \mathrm{U}6^{-}\ \mathrm{O}6^{-}$ & $1$ & $[1,2]$ & $\leq 2$ \\ \hline
$6.655$ & $\mathrm{O}1^{-}\ \mathrm{O}2^{-}\ \mathrm{U}1^{-}\ \mathrm{U}2^{-}\ \mathrm{O}3^{+}\ \mathrm{O}4^{+}\ \mathrm{O}5^{+}\ \mathrm{U}6^{+}\ \mathrm{U}5^{+}\ \mathrm{O}6^{+}\ \mathrm{U}4^{+}\ \mathrm{U}3^{+}$ & $\mathrm{O}2^{-}\ \mathrm{U}2^{-}\ \mathrm{U}6^{+}\ \mathrm{O}6^{+}$ & $1$ & $[1,2]$ & $\leq 2$ \\ \hline
$6.1892$ & $\mathrm{O}1^{-}\ \mathrm{O}2^{-}\ \mathrm{U}1^{-}\ \mathrm{U}2^{-}\ \mathrm{O}3^{-}\ \mathrm{U}4^{-}\ \mathrm{O}5^{-}\ \mathrm{U}6^{-}\ \mathrm{U}5^{-}\ \mathrm{O}6^{-}\ \mathrm{O}4^{-}\ \mathrm{U}3^{-}$ & $\mathrm{O}2^{-}\ \mathrm{U}2^{-}\ \mathrm{U}6^{-}\ \mathrm{O}6^{-}$ & $1$ & $2$ & $2$ \\ \hline
$6.1894$ & $\mathrm{O}1^{-}\ \mathrm{O}2^{-}\ \mathrm{U}1^{-}\ \mathrm{U}2^{-}\ \mathrm{O}3^{-}\ \mathrm{U}4^{-}\ \mathrm{O}5^{+}\ \mathrm{U}6^{+}\ \mathrm{U}5^{+}\ \mathrm{O}6^{+}\ \mathrm{O}4^{-}\ \mathrm{U}3^{-}$ & $\mathrm{O}2^{-}\ \mathrm{U}2^{-}\ \mathrm{U}6^{+}\ \mathrm{O}6^{+}$ & $1$ & $[1,2]$ & $\leq 2$ \\ \hline
$6.1895$ & $\mathrm{O}1^{-}\ \mathrm{O}2^{-}\ \mathrm{U}1^{-}\ \mathrm{U}2^{-}\ \mathrm{O}3^{-}\ \mathrm{U}4^{+}\ \mathrm{O}5^{-}\ \mathrm{U}6^{-}\ \mathrm{U}5^{-}\ \mathrm{O}6^{-}\ \mathrm{O}4^{+}\ \mathrm{U}3^{-}$ & $\mathrm{O}2^{-}\ \mathrm{U}2^{-}\ \mathrm{U}6^{-}\ \mathrm{O}6^{-}$ & $1$ & $[1,2]$ & $\leq 2$ \\ \hline
$6.1897$ & $\mathrm{O}1^{-}\ \mathrm{O}2^{-}\ \mathrm{U}1^{-}\ \mathrm{U}2^{-}\ \mathrm{O}3^{-}\ \mathrm{U}4^{+}\ \mathrm{O}5^{+}\ \mathrm{U}6^{+}\ \mathrm{U}5^{+}\ \mathrm{O}6^{+}\ \mathrm{O}4^{+}\ \mathrm{U}3^{-}$ & $\mathrm{O}2^{-}\ \mathrm{U}2^{-}\ \mathrm{U}6^{+}\ \mathrm{O}6^{+}$ & $1$ & $[1,2]$ & $\leq 2$ \\ \hline
$6.1898$ & $\mathrm{O}1^{-}\ \mathrm{O}2^{-}\ \mathrm{U}1^{-}\ \mathrm{U}2^{-}\ \mathrm{O}3^{+}\ \mathrm{U}4^{-}\ \mathrm{O}5^{-}\ \mathrm{U}6^{-}\ \mathrm{U}5^{-}\ \mathrm{O}6^{-}\ \mathrm{O}4^{-}\ \mathrm{U}3^{+}$ & $\mathrm{O}2^{-}\ \mathrm{U}2^{-}\ \mathrm{U}6^{-}\ \mathrm{O}6^{-}$ & $1$ & $[1,2]$ & $\leq 2$ \\ \hline
$6.1900$ & $\mathrm{O}1^{-}\ \mathrm{O}2^{-}\ \mathrm{U}1^{-}\ \mathrm{U}2^{-}\ \mathrm{O}3^{+}\ \mathrm{U}4^{-}\ \mathrm{O}5^{+}\ \mathrm{U}6^{+}\ \mathrm{U}5^{+}\ \mathrm{O}6^{+}\ \mathrm{O}4^{-}\ \mathrm{U}3^{+}$ & $\mathrm{O}2^{-}\ \mathrm{U}2^{-}\ \mathrm{U}6^{+}\ \mathrm{O}6^{+}$ & $1$ & $[1,2]$ & $\leq 2$ \\ \hline
$6.1901$ & $\mathrm{O}1^{-}\ \mathrm{O}2^{-}\ \mathrm{U}1^{-}\ \mathrm{U}2^{-}\ \mathrm{O}3^{+}\ \mathrm{U}4^{+}\ \mathrm{O}5^{-}\ \mathrm{U}6^{-}\ \mathrm{U}5^{-}\ \mathrm{O}6^{-}\ \mathrm{O}4^{+}\ \mathrm{U}3^{+}$ & $\mathrm{O}2^{-}\ \mathrm{U}2^{-}\ \mathrm{U}6^{-}\ \mathrm{O}6^{-}$ & $1$ & $[1,2]$ & $\leq 2$ \\ \hline
$6.1903$ & $\mathrm{O}1^{-}\ \mathrm{O}2^{-}\ \mathrm{U}1^{-}\ \mathrm{U}2^{-}\ \mathrm{O}3^{+}\ \mathrm{U}4^{+}\ \mathrm{O}5^{+}\ \mathrm{U}6^{+}\ \mathrm{U}5^{+}\ \mathrm{O}6^{+}\ \mathrm{O}4^{+}\ \mathrm{U}3^{+}$ & $\mathrm{O}2^{-}\ \mathrm{U}2^{-}\ \mathrm{U}6^{+}\ \mathrm{O}6^{+}$ & $1$ & $[1,2]$ & $\leq 2$ \\ \hline
$6.2601$ & $\mathrm{O}1^{-}\ \mathrm{O}2^{-}\ \mathrm{U}1^{-}\ \mathrm{U}2^{-}\ \mathrm{U}3^{-}\ \mathrm{O}4^{-}\ \mathrm{O}5^{-}\ \mathrm{U}6^{-}\ \mathrm{U}5^{-}\ \mathrm{O}6^{-}\ \mathrm{U}4^{-}\ \mathrm{O}3^{-}$ & $\mathrm{O}2^{-}\ \mathrm{U}2^{-}\ \mathrm{U}6^{-}\ \mathrm{O}6^{-}$ & $1$ & $2$ & $2$ \\ \hline
$6.2603$ & $\mathrm{O}1^{-}\ \mathrm{O}2^{-}\ \mathrm{U}1^{-}\ \mathrm{U}2^{-}\ \mathrm{U}3^{-}\ \mathrm{O}4^{-}\ \mathrm{O}5^{+}\ \mathrm{U}6^{+}\ \mathrm{U}5^{+}\ \mathrm{O}6^{+}\ \mathrm{U}4^{-}\ \mathrm{O}3^{-}$ & $\mathrm{O}2^{-}\ \mathrm{U}2^{-}\ \mathrm{U}6^{+}\ \mathrm{O}6^{+}$ & $1$ & $[1,2]$ & $\leq 2$ \\ \hline
$6.2604$ & $\mathrm{O}1^{-}\ \mathrm{O}2^{-}\ \mathrm{U}1^{-}\ \mathrm{U}2^{-}\ \mathrm{U}3^{-}\ \mathrm{O}4^{+}\ \mathrm{O}5^{-}\ \mathrm{U}6^{-}\ \mathrm{U}5^{-}\ \mathrm{O}6^{-}\ \mathrm{U}4^{+}\ \mathrm{O}3^{-}$ & $\mathrm{O}2^{-}\ \mathrm{U}2^{-}\ \mathrm{U}6^{-}\ \mathrm{O}6^{-}$ & $1$ & $[1,2]$ & $\leq 2$ \\ \hline
$6.2606$ & $\mathrm{O}1^{-}\ \mathrm{O}2^{-}\ \mathrm{U}1^{-}\ \mathrm{U}2^{-}\ \mathrm{U}3^{-}\ \mathrm{O}4^{+}\ \mathrm{O}5^{+}\ \mathrm{U}6^{+}\ \mathrm{U}5^{+}\ \mathrm{O}6^{+}\ \mathrm{U}4^{+}\ \mathrm{O}3^{-}$ & $\mathrm{O}2^{-}\ \mathrm{U}2^{-}\ \mathrm{U}6^{+}\ \mathrm{O}6^{+}$ & $1$ & $[1,2]$ & $\leq 2$ \\ \hline
$6.2607$ & $\mathrm{O}1^{-}\ \mathrm{O}2^{-}\ \mathrm{U}1^{-}\ \mathrm{U}2^{-}\ \mathrm{U}3^{+}\ \mathrm{O}4^{-}\ \mathrm{O}5^{-}\ \mathrm{U}6^{-}\ \mathrm{U}5^{-}\ \mathrm{O}6^{-}\ \mathrm{U}4^{-}\ \mathrm{O}3^{+}$ & $\mathrm{O}2^{-}\ \mathrm{U}2^{-}\ \mathrm{U}6^{-}\ \mathrm{O}6^{-}$ & $1$ & $[1,2]$ & $\leq 2$ \\ \hline
$6.2609$ & $\mathrm{O}1^{-}\ \mathrm{O}2^{-}\ \mathrm{U}1^{-}\ \mathrm{U}2^{-}\ \mathrm{U}3^{+}\ \mathrm{O}4^{-}\ \mathrm{O}5^{+}\ \mathrm{U}6^{+}\ \mathrm{U}5^{+}\ \mathrm{O}6^{+}\ \mathrm{U}4^{-}\ \mathrm{O}3^{+}$ & $\mathrm{O}2^{-}\ \mathrm{U}2^{-}\ \mathrm{U}6^{+}\ \mathrm{O}6^{+}$ & $1$ & $[1,2]$ & $\leq 2$ \\ \hline
$6.2610$ & $\mathrm{O}1^{-}\ \mathrm{O}2^{-}\ \mathrm{U}1^{-}\ \mathrm{U}2^{-}\ \mathrm{U}3^{+}\ \mathrm{O}4^{+}\ \mathrm{O}5^{-}\ \mathrm{U}6^{-}\ \mathrm{U}5^{-}\ \mathrm{O}6^{-}\ \mathrm{U}4^{+}\ \mathrm{O}3^{+}$ & $\mathrm{O}2^{-}\ \mathrm{U}2^{-}\ \mathrm{U}6^{-}\ \mathrm{O}6^{-}$ & $1$ & $[1,2]$ & $\leq 2$ \\ \hline
$6.2612$ & $\mathrm{O}1^{-}\ \mathrm{O}2^{-}\ \mathrm{U}1^{-}\ \mathrm{U}2^{-}\ \mathrm{U}3^{+}\ \mathrm{O}4^{+}\ \mathrm{O}5^{+}\ \mathrm{U}6^{+}\ \mathrm{U}5^{+}\ \mathrm{O}6^{+}\ \mathrm{U}4^{+}\ \mathrm{O}3^{+}$ & $\mathrm{O}2^{-}\ \mathrm{U}2^{-}\ \mathrm{U}6^{+}\ \mathrm{O}6^{+}$ & $1$ & $[1,2]$ & $\leq 2$ \\ \hline
$6.2684$ & $\mathrm{O}1^{-}\ \mathrm{O}2^{-}\ \mathrm{U}1^{-}\ \mathrm{U}2^{-}\ \mathrm{U}3^{-}\ \mathrm{O}4^{-}\ \mathrm{U}5^{-}\ \mathrm{U}4^{-}\ \mathrm{O}5^{-}\ \mathrm{U}6^{+}\ \mathrm{O}3^{-}\ \mathrm{O}6^{+}$ & $\mathrm{O}2^{-}\ \mathrm{U}2^{-}\ \mathrm{U}5^{-}\ \mathrm{O}5^{-}\ \mathrm{U}6^{+}\ \mathrm{O}6^{+}$ & $1$ & $[1,2]$ & $\leq 2$ \\ \hline
$6.2690$ & $\mathrm{O}1^{-}\ \mathrm{O}2^{-}\ \mathrm{U}1^{-}\ \mathrm{U}2^{-}\ \mathrm{U}3^{-}\ \mathrm{O}4^{+}\ \mathrm{U}5^{+}\ \mathrm{U}4^{+}\ \mathrm{O}5^{+}\ \mathrm{U}6^{+}\ \mathrm{O}3^{-}\ \mathrm{O}6^{+}$ & $\mathrm{O}2^{-}\ \mathrm{U}2^{-}\ \mathrm{U}5^{+}\ \mathrm{O}5^{+}\ \mathrm{U}6^{+}\ \mathrm{O}6^{+}$ & $1$ & $[1,2]$ & $\leq 2$ \\ \hline
$6.2691$ & $\mathrm{O}1^{-}\ \mathrm{O}2^{-}\ \mathrm{U}1^{-}\ \mathrm{U}2^{-}\ \mathrm{U}3^{+}\ \mathrm{O}4^{-}\ \mathrm{U}5^{-}\ \mathrm{U}4^{-}\ \mathrm{O}5^{-}\ \mathrm{U}6^{-}\ \mathrm{O}3^{+}\ \mathrm{O}6^{-}$ & $\mathrm{O}2^{-}\ \mathrm{U}2^{-}\ \mathrm{U}5^{-}\ \mathrm{O}5^{-}\ \mathrm{U}6^{-}\ \mathrm{O}6^{-}$ & $1$ & $[1,2]$ & $\leq 2$ \\ \hline
$6.2697$ & $\mathrm{O}1^{-}\ \mathrm{O}2^{-}\ \mathrm{U}1^{-}\ \mathrm{U}2^{-}\ \mathrm{U}3^{+}\ \mathrm{O}4^{+}\ \mathrm{U}5^{+}\ \mathrm{U}4^{+}\ \mathrm{O}5^{+}\ \mathrm{U}6^{-}\ \mathrm{O}3^{+}\ \mathrm{O}6^{-}$ & $\mathrm{O}2^{-}\ \mathrm{U}2^{-}\ \mathrm{U}5^{+}\ \mathrm{O}5^{+}\ \mathrm{U}6^{-}\ \mathrm{O}6^{-}$ & $1$ & $[1,2]$ & $\leq 2$ \\ \hline
$6.2943$ & $\mathrm{O}1^{-}\ \mathrm{O}2^{-}\ \mathrm{U}1^{-}\ \mathrm{U}2^{-}\ \mathrm{U}3^{-}\ \mathrm{U}4^{+}\ \mathrm{O}3^{-}\ \mathrm{O}5^{-}\ \mathrm{U}6^{-}\ \mathrm{U}5^{-}\ \mathrm{O}6^{-}\ \mathrm{O}4^{+}$ & $\mathrm{O}2^{-}\ \mathrm{U}2^{-}\ \mathrm{U}6^{-}\ \mathrm{O}6^{-}$ & $1$ & $[1,2]$ & $\leq 2$ \\ \hline
$6.2946$ & $\mathrm{O}1^{-}\ \mathrm{O}2^{-}\ \mathrm{U}1^{-}\ \mathrm{U}2^{-}\ \mathrm{U}3^{-}\ \mathrm{U}4^{+}\ \mathrm{O}3^{-}\ \mathrm{O}5^{+}\ \mathrm{U}6^{+}\ \mathrm{U}5^{+}\ \mathrm{O}6^{+}\ \mathrm{O}4^{+}$ & $\mathrm{O}2^{-}\ \mathrm{U}2^{-}\ \mathrm{U}6^{+}\ \mathrm{O}6^{+}$ & $1$ & $[1,2]$ & $\leq 2$ \\ \hline
$6.2947$ & $\mathrm{O}1^{-}\ \mathrm{O}2^{-}\ \mathrm{U}1^{-}\ \mathrm{U}2^{-}\ \mathrm{U}3^{+}\ \mathrm{U}4^{-}\ \mathrm{O}3^{+}\ \mathrm{O}5^{-}\ \mathrm{U}6^{-}\ \mathrm{U}5^{-}\ \mathrm{O}6^{-}\ \mathrm{O}4^{-}$ & $\mathrm{O}2^{-}\ \mathrm{U}2^{-}\ \mathrm{U}6^{-}\ \mathrm{O}6^{-}$ & $1$ & $[1,2]$ & $\leq 2$ \\ \hline
$6.2950$ & $\mathrm{O}1^{-}\ \mathrm{O}2^{-}\ \mathrm{U}1^{-}\ \mathrm{U}2^{-}\ \mathrm{U}3^{+}\ \mathrm{U}4^{-}\ \mathrm{O}3^{+}\ \mathrm{O}5^{+}\ \mathrm{U}6^{+}\ \mathrm{U}5^{+}\ \mathrm{O}6^{+}\ \mathrm{O}4^{-}$ & $\mathrm{O}2^{-}\ \mathrm{U}2^{-}\ \mathrm{U}6^{+}\ \mathrm{O}6^{+}$ & $1$ & $[1,2]$ & $\leq 2$ \\ \hline
$6.3097$ & $\mathrm{O}1^{-}\ \mathrm{O}2^{-}\ \mathrm{U}1^{-}\ \mathrm{U}2^{-}\ \mathrm{U}3^{-}\ \mathrm{U}4^{-}\ \mathrm{O}5^{-}\ \mathrm{U}6^{-}\ \mathrm{U}5^{-}\ \mathrm{O}6^{-}\ \mathrm{O}4^{-}\ \mathrm{O}3^{-}$ & $\mathrm{O}2^{-}\ \mathrm{U}2^{-}\ \mathrm{U}6^{-}\ \mathrm{O}6^{-}$ & $0$ & $2$ & $2$ \\ \hline
$6.3099$ & $\mathrm{O}1^{-}\ \mathrm{O}2^{-}\ \mathrm{U}1^{-}\ \mathrm{U}2^{-}\ \mathrm{U}3^{-}\ \mathrm{U}4^{-}\ \mathrm{O}5^{+}\ \mathrm{U}6^{+}\ \mathrm{U}5^{+}\ \mathrm{O}6^{+}\ \mathrm{O}4^{-}\ \mathrm{O}3^{-}$ & $\mathrm{O}2^{-}\ \mathrm{U}2^{-}\ \mathrm{U}6^{+}\ \mathrm{O}6^{+}$ & $0$ & $[1,2]$ & $\leq 2$ \\ \hline
$6.3100$ & $\mathrm{O}1^{-}\ \mathrm{O}2^{-}\ \mathrm{U}1^{-}\ \mathrm{U}2^{-}\ \mathrm{U}3^{+}\ \mathrm{U}4^{+}\ \mathrm{O}5^{-}\ \mathrm{U}6^{-}\ \mathrm{U}5^{-}\ \mathrm{O}6^{-}\ \mathrm{O}4^{+}\ \mathrm{O}3^{+}$ & $\mathrm{O}2^{-}\ \mathrm{U}2^{-}\ \mathrm{U}6^{-}\ \mathrm{O}6^{-}$ & $0$ & $[1,2]$ & $\leq 2$ \\ \hline
$6.3102$ & $\mathrm{O}1^{-}\ \mathrm{O}2^{-}\ \mathrm{U}1^{-}\ \mathrm{U}2^{-}\ \mathrm{U}3^{+}\ \mathrm{U}4^{+}\ \mathrm{O}5^{+}\ \mathrm{U}6^{+}\ \mathrm{U}5^{+}\ \mathrm{O}6^{+}\ \mathrm{O}4^{+}\ \mathrm{O}3^{+}$ & $\mathrm{O}2^{-}\ \mathrm{U}2^{-}\ \mathrm{U}6^{+}\ \mathrm{O}6^{+}$ & $0$ & $[1,2]$ & $\leq 2$ \\ \hline
$6.14824$ & $\mathrm{O}1^{-}\ \mathrm{O}2^{-}\ \mathrm{U}1^{-}\ \mathrm{O}3^{-}\ \mathrm{O}4^{+}\ \mathrm{U}3^{-}\ \mathrm{O}5^{-}\ \mathrm{U}6^{-}\ \mathrm{U}5^{-}\ \mathrm{O}6^{-}\ \mathrm{U}4^{+}\ \mathrm{U}2^{-}$ & $\mathrm{U}6^{-}\ \mathrm{O}6^{-}$ & $1$ & $[1,2]$ & $\leq 2$ \\ \hline
$6.14827$ & $\mathrm{O}1^{-}\ \mathrm{O}2^{-}\ \mathrm{U}1^{-}\ \mathrm{O}3^{-}\ \mathrm{O}4^{+}\ \mathrm{U}3^{-}\ \mathrm{O}5^{+}\ \mathrm{U}6^{+}\ \mathrm{U}5^{+}\ \mathrm{O}6^{+}\ \mathrm{U}4^{+}\ \mathrm{U}2^{-}$ & $\mathrm{U}6^{+}\ \mathrm{O}6^{+}$ & $1$ & $[1,2]$ & $\leq 2$ \\ \hline
$6.14828$ & $\mathrm{O}1^{-}\ \mathrm{O}2^{-}\ \mathrm{U}1^{-}\ \mathrm{O}3^{+}\ \mathrm{O}4^{-}\ \mathrm{U}3^{+}\ \mathrm{O}5^{-}\ \mathrm{U}6^{-}\ \mathrm{U}5^{-}\ \mathrm{O}6^{-}\ \mathrm{U}4^{-}\ \mathrm{U}2^{-}$ & $\mathrm{U}6^{-}\ \mathrm{O}6^{-}$ & $1$ & $[1,2]$ & $\leq 2$ \\ \hline
$6.14831$ & $\mathrm{O}1^{-}\ \mathrm{O}2^{-}\ \mathrm{U}1^{-}\ \mathrm{O}3^{+}\ \mathrm{O}4^{-}\ \mathrm{U}3^{+}\ \mathrm{O}5^{+}\ \mathrm{U}6^{+}\ \mathrm{U}5^{+}\ \mathrm{O}6^{+}\ \mathrm{U}4^{-}\ \mathrm{U}2^{-}$ & $\mathrm{U}6^{+}\ \mathrm{O}6^{+}$ & $1$ & $[1,2]$ & $\leq 2$ \\ \hline
$6.14840$ & $\mathrm{O}1^{-}\ \mathrm{O}2^{+}\ \mathrm{U}1^{-}\ \mathrm{O}3^{-}\ \mathrm{O}4^{+}\ \mathrm{U}3^{-}\ \mathrm{O}5^{-}\ \mathrm{U}6^{-}\ \mathrm{U}5^{-}\ \mathrm{O}6^{-}\ \mathrm{U}4^{+}\ \mathrm{U}2^{+}$ & $\mathrm{U}6^{-}\ \mathrm{O}6^{-}$ & $1$ & $1$ & $1$ \\ \hline
$6.14843$ & $\mathrm{O}1^{-}\ \mathrm{O}2^{+}\ \mathrm{U}1^{-}\ \mathrm{O}3^{-}\ \mathrm{O}4^{+}\ \mathrm{U}3^{-}\ \mathrm{O}5^{+}\ \mathrm{U}6^{+}\ \mathrm{U}5^{+}\ \mathrm{O}6^{+}\ \mathrm{U}4^{+}\ \mathrm{U}2^{+}$ & $\mathrm{U}6^{+}\ \mathrm{O}6^{+}$ & $1$ & $1$ & $1$ \\ \hline
$6.14844$ & $\mathrm{O}1^{-}\ \mathrm{O}2^{+}\ \mathrm{U}1^{-}\ \mathrm{O}3^{+}\ \mathrm{O}4^{-}\ \mathrm{U}3^{+}\ \mathrm{O}5^{-}\ \mathrm{U}6^{-}\ \mathrm{U}5^{-}\ \mathrm{O}6^{-}\ \mathrm{U}4^{-}\ \mathrm{U}2^{+}$ & $\mathrm{U}6^{-}\ \mathrm{O}6^{-}$ & $1$ & $1$ & $1$ \\ \hline
$6.14847$ & $\mathrm{O}1^{-}\ \mathrm{O}2^{+}\ \mathrm{U}1^{-}\ \mathrm{O}3^{+}\ \mathrm{O}4^{-}\ \mathrm{U}3^{+}\ \mathrm{O}5^{+}\ \mathrm{U}6^{+}\ \mathrm{U}5^{+}\ \mathrm{O}6^{+}\ \mathrm{U}4^{-}\ \mathrm{U}2^{+}$ & $\mathrm{U}6^{+}\ \mathrm{O}6^{+}$ & $1$ & $1$ & $1$ \\ \hline
$6.15981$ & $\mathrm{O}1^{-}\ \mathrm{O}2^{-}\ \mathrm{U}1^{-}\ \mathrm{O}3^{-}\ \mathrm{O}4^{+}\ \mathrm{U}3^{-}\ \mathrm{U}5^{-}\ \mathrm{O}6^{-}\ \mathrm{O}5^{-}\ \mathrm{U}6^{-}\ \mathrm{U}4^{+}\ \mathrm{U}2^{-}$ & $\mathrm{O}6^{-}\ \mathrm{U}6^{-}$ & $1$ & $[1,2]$ & $\leq 2$ \\ \hline
$6.15984$ & $\mathrm{O}1^{-}\ \mathrm{O}2^{-}\ \mathrm{U}1^{-}\ \mathrm{O}3^{-}\ \mathrm{O}4^{+}\ \mathrm{U}3^{-}\ \mathrm{U}5^{+}\ \mathrm{O}6^{+}\ \mathrm{O}5^{+}\ \mathrm{U}6^{+}\ \mathrm{U}4^{+}\ \mathrm{U}2^{-}$ & $\mathrm{O}6^{+}\ \mathrm{U}6^{+}$ & $1$ & $1$ & $\leq 2$ \\ \hline
$6.15985$ & $\mathrm{O}1^{-}\ \mathrm{O}2^{-}\ \mathrm{U}1^{-}\ \mathrm{O}3^{+}\ \mathrm{O}4^{-}\ \mathrm{U}3^{+}\ \mathrm{U}5^{-}\ \mathrm{O}6^{-}\ \mathrm{O}5^{-}\ \mathrm{U}6^{-}\ \mathrm{U}4^{-}\ \mathrm{U}2^{-}$ & $\mathrm{O}6^{-}\ \mathrm{U}6^{-}$ & $1$ & $[1,2]$ & $\leq 2$ \\ \hline
$6.15988$ & $\mathrm{O}1^{-}\ \mathrm{O}2^{-}\ \mathrm{U}1^{-}\ \mathrm{O}3^{+}\ \mathrm{O}4^{-}\ \mathrm{U}3^{+}\ \mathrm{U}5^{+}\ \mathrm{O}6^{+}\ \mathrm{O}5^{+}\ \mathrm{U}6^{+}\ \mathrm{U}4^{-}\ \mathrm{U}2^{-}$ & $\mathrm{O}6^{+}\ \mathrm{U}6^{+}$ & $1$ & $1$ & $\leq 2$ \\ \hline
$6.15997$ & $\mathrm{O}1^{-}\ \mathrm{O}2^{+}\ \mathrm{U}1^{-}\ \mathrm{O}3^{-}\ \mathrm{O}4^{+}\ \mathrm{U}3^{-}\ \mathrm{U}5^{-}\ \mathrm{O}6^{-}\ \mathrm{O}5^{-}\ \mathrm{U}6^{-}\ \mathrm{U}4^{+}\ \mathrm{U}2^{+}$ & $\mathrm{O}6^{-}\ \mathrm{U}6^{-}$ & $1$ & $1$ & $1$ \\ \hline
$6.16000$ & $\mathrm{O}1^{-}\ \mathrm{O}2^{+}\ \mathrm{U}1^{-}\ \mathrm{O}3^{-}\ \mathrm{O}4^{+}\ \mathrm{U}3^{-}\ \mathrm{U}5^{+}\ \mathrm{O}6^{+}\ \mathrm{O}5^{+}\ \mathrm{U}6^{+}\ \mathrm{U}4^{+}\ \mathrm{U}2^{+}$ & $\mathrm{O}6^{+}\ \mathrm{U}6^{+}$ & $1$ & $1$ & $1$ \\ \hline
$6.16001$ & $\mathrm{O}1^{-}\ \mathrm{O}2^{+}\ \mathrm{U}1^{-}\ \mathrm{O}3^{+}\ \mathrm{O}4^{-}\ \mathrm{U}3^{+}\ \mathrm{U}5^{-}\ \mathrm{O}6^{-}\ \mathrm{O}5^{-}\ \mathrm{U}6^{-}\ \mathrm{U}4^{-}\ \mathrm{U}2^{+}$ & $\mathrm{O}6^{-}\ \mathrm{U}6^{-}$ & $1$ & $1$ & $1$ \\ \hline
$6.16004$ & $\mathrm{O}1^{-}\ \mathrm{O}2^{+}\ \mathrm{U}1^{-}\ \mathrm{O}3^{+}\ \mathrm{O}4^{-}\ \mathrm{U}3^{+}\ \mathrm{U}5^{+}\ \mathrm{O}6^{+}\ \mathrm{O}5^{+}\ \mathrm{U}6^{+}\ \mathrm{U}4^{-}\ \mathrm{U}2^{+}$ & $\mathrm{O}6^{+}\ \mathrm{U}6^{+}$ & $1$ & $1$ & $1$ \\ \hline
$6.22566$ & $\mathrm{O}1^{-}\ \mathrm{O}2^{-}\ \mathrm{U}1^{-}\ \mathrm{O}3^{-}\ \mathrm{O}4^{-}\ \mathrm{O}5^{-}\ \mathrm{U}6^{-}\ \mathrm{U}5^{-}\ \mathrm{O}6^{-}\ \mathrm{U}4^{-}\ \mathrm{U}3^{-}\ \mathrm{U}2^{-}$ & $\mathrm{U}6^{-}\ \mathrm{O}6^{-}$ & $1$ & $2$ & $2$ \\ \hline
$6.22567$ & $\mathrm{O}1^{-}\ \mathrm{O}2^{-}\ \mathrm{U}1^{-}\ \mathrm{O}3^{-}\ \mathrm{O}4^{-}\ \mathrm{O}5^{+}\ \mathrm{U}6^{-}\ \mathrm{U}5^{+}\ \mathrm{O}6^{-}\ \mathrm{U}4^{-}\ \mathrm{U}3^{-}\ \mathrm{U}2^{-}$ & $\mathrm{U}6^{-}\ \mathrm{O}6^{-}$ & $1$ & $1$ & $1$ \\ \hline
$6.22568$ & $\mathrm{O}1^{-}\ \mathrm{O}2^{-}\ \mathrm{U}1^{-}\ \mathrm{O}3^{-}\ \mathrm{O}4^{-}\ \mathrm{O}5^{+}\ \mathrm{U}6^{+}\ \mathrm{U}5^{+}\ \mathrm{O}6^{+}\ \mathrm{U}4^{-}\ \mathrm{U}3^{-}\ \mathrm{U}2^{-}$ & $\mathrm{U}6^{+}\ \mathrm{O}6^{+}$ & $1$ & $[1,2]$ & $\leq 2$ \\ \hline
$6.22569$ & $\mathrm{O}1^{-}\ \mathrm{O}2^{-}\ \mathrm{U}1^{-}\ \mathrm{O}3^{+}\ \mathrm{O}4^{+}\ \mathrm{O}5^{-}\ \mathrm{U}6^{-}\ \mathrm{U}5^{-}\ \mathrm{O}6^{-}\ \mathrm{U}4^{+}\ \mathrm{U}3^{+}\ \mathrm{U}2^{-}$ & $\mathrm{U}6^{-}\ \mathrm{O}6^{-}$ & $1$ & $[1,2]$ & $\leq 2$ \\ \hline
$6.22570$ & $\mathrm{O}1^{-}\ \mathrm{O}2^{-}\ \mathrm{U}1^{-}\ \mathrm{O}3^{+}\ \mathrm{O}4^{+}\ \mathrm{O}5^{-}\ \mathrm{U}6^{+}\ \mathrm{U}5^{-}\ \mathrm{O}6^{+}\ \mathrm{U}4^{+}\ \mathrm{U}3^{+}\ \mathrm{U}2^{-}$ & $\mathrm{U}6^{+}\ \mathrm{O}6^{+}$ & $1$ & $1$ & $1$ \\ \hline
$6.25946$ & $\mathrm{O}1^{-}\ \mathrm{O}2^{-}\ \mathrm{U}1^{-}\ \mathrm{O}3^{-}\ \mathrm{O}4^{+}\ \mathrm{U}5^{+}\ \mathrm{U}4^{+}\ \mathrm{O}5^{+}\ \mathrm{O}6^{+}\ \mathrm{U}3^{-}\ \mathrm{U}6^{+}\ \mathrm{U}2^{-}$ & $\mathrm{U}5^{+}\ \mathrm{O}5^{+}\ \mathrm{O}6^{+}\ \mathrm{U}6^{+}$ & $1$ & $[1,2]$ & $\leq 2$ \\ \hline
$6.25947$ & $\mathrm{O}1^{-}\ \mathrm{O}2^{-}\ \mathrm{U}1^{-}\ \mathrm{O}3^{+}\ \mathrm{O}4^{-}\ \mathrm{U}5^{-}\ \mathrm{U}4^{-}\ \mathrm{O}5^{-}\ \mathrm{O}6^{-}\ \mathrm{U}3^{+}\ \mathrm{U}6^{-}\ \mathrm{U}2^{-}$ & $\mathrm{U}5^{-}\ \mathrm{O}5^{-}\ \mathrm{O}6^{-}\ \mathrm{U}6^{-}$ & $1$ & $[1,2]$ & $\leq 2$ \\ \hline
$6.25953$ & $\mathrm{O}1^{-}\ \mathrm{O}2^{-}\ \mathrm{U}1^{-}\ \mathrm{O}3^{+}\ \mathrm{O}4^{+}\ \mathrm{U}5^{+}\ \mathrm{U}4^{+}\ \mathrm{O}5^{+}\ \mathrm{O}6^{-}\ \mathrm{U}3^{+}\ \mathrm{U}6^{-}\ \mathrm{U}2^{-}$ & $\mathrm{U}5^{+}\ \mathrm{O}5^{+}\ \mathrm{O}6^{-}\ \mathrm{U}6^{-}$ & $1$ & $[1,2]$ & $\leq 2$ \\ \hline
$6.25956$ & $\mathrm{O}1^{-}\ \mathrm{O}2^{+}\ \mathrm{U}1^{-}\ \mathrm{O}3^{-}\ \mathrm{O}4^{-}\ \mathrm{U}5^{-}\ \mathrm{U}4^{-}\ \mathrm{O}5^{-}\ \mathrm{O}6^{+}\ \mathrm{U}3^{-}\ \mathrm{U}6^{+}\ \mathrm{U}2^{+}$ & $\mathrm{U}5^{-}\ \mathrm{O}5^{-}\ \mathrm{O}6^{+}\ \mathrm{U}6^{+}$ & $1$ & $1$ & $1$ \\ \hline
$6.25962$ & $\mathrm{O}1^{-}\ \mathrm{O}2^{+}\ \mathrm{U}1^{-}\ \mathrm{O}3^{-}\ \mathrm{O}4^{+}\ \mathrm{U}5^{+}\ \mathrm{U}4^{+}\ \mathrm{O}5^{+}\ \mathrm{O}6^{+}\ \mathrm{U}3^{-}\ \mathrm{U}6^{+}\ \mathrm{U}2^{+}$ & $\mathrm{U}5^{+}\ \mathrm{O}5^{+}\ \mathrm{O}6^{+}\ \mathrm{U}6^{+}$ & $1$ & $1$ & $1$ \\ \hline
$6.25963$ & $\mathrm{O}1^{-}\ \mathrm{O}2^{+}\ \mathrm{U}1^{-}\ \mathrm{O}3^{+}\ \mathrm{O}4^{-}\ \mathrm{U}5^{-}\ \mathrm{U}4^{-}\ \mathrm{O}5^{-}\ \mathrm{O}6^{-}\ \mathrm{U}3^{+}\ \mathrm{U}6^{-}\ \mathrm{U}2^{+}$ & $\mathrm{U}5^{-}\ \mathrm{O}5^{-}\ \mathrm{O}6^{-}\ \mathrm{U}6^{-}$ & $1$ & $1$ & $1$ \\ \hline
$6.25969$ & $\mathrm{O}1^{-}\ \mathrm{O}2^{+}\ \mathrm{U}1^{-}\ \mathrm{O}3^{+}\ \mathrm{O}4^{+}\ \mathrm{U}5^{+}\ \mathrm{U}4^{+}\ \mathrm{O}5^{+}\ \mathrm{O}6^{-}\ \mathrm{U}3^{+}\ \mathrm{U}6^{-}\ \mathrm{U}2^{+}$ & $\mathrm{U}5^{+}\ \mathrm{O}5^{+}\ \mathrm{O}6^{-}\ \mathrm{U}6^{-}$ & $1$ & $1$ & $1$ \\ \hline
$6.28589$ & $\mathrm{O}1^{-}\ \mathrm{O}2^{-}\ \mathrm{U}1^{-}\ \mathrm{O}3^{-}\ \mathrm{O}4^{-}\ \mathrm{U}5^{-}\ \mathrm{O}6^{-}\ \mathrm{O}5^{-}\ \mathrm{U}6^{-}\ \mathrm{U}4^{-}\ \mathrm{U}3^{-}\ \mathrm{U}2^{-}$ & $\mathrm{O}6^{-}\ \mathrm{U}6^{-}$ & $1$ & $2$ & $2$ \\ \hline
$6.28590$ & $\mathrm{O}1^{-}\ \mathrm{O}2^{-}\ \mathrm{U}1^{-}\ \mathrm{O}3^{-}\ \mathrm{O}4^{-}\ \mathrm{U}5^{-}\ \mathrm{O}6^{+}\ \mathrm{O}5^{-}\ \mathrm{U}6^{+}\ \mathrm{U}4^{-}\ \mathrm{U}3^{-}\ \mathrm{U}2^{-}$ & $\mathrm{O}6^{+}\ \mathrm{U}6^{+}$ & $1$ & $1$ & $1$ \\ \hline
$6.28591$ & $\mathrm{O}1^{-}\ \mathrm{O}2^{-}\ \mathrm{U}1^{-}\ \mathrm{O}3^{-}\ \mathrm{O}4^{-}\ \mathrm{U}5^{+}\ \mathrm{O}6^{-}\ \mathrm{O}5^{+}\ \mathrm{U}6^{-}\ \mathrm{U}4^{-}\ \mathrm{U}3^{-}\ \mathrm{U}2^{-}$ & $\mathrm{O}6^{-}\ \mathrm{U}6^{-}$ & $1$ & $1$ & $1$ \\ \hline
$6.28592$ & $\mathrm{O}1^{-}\ \mathrm{O}2^{-}\ \mathrm{U}1^{-}\ \mathrm{O}3^{-}\ \mathrm{O}4^{-}\ \mathrm{U}5^{+}\ \mathrm{O}6^{+}\ \mathrm{O}5^{+}\ \mathrm{U}6^{+}\ \mathrm{U}4^{-}\ \mathrm{U}3^{-}\ \mathrm{U}2^{-}$ & $\mathrm{O}6^{+}\ \mathrm{U}6^{+}$ & $1$ & $1$ & $\leq 2$ \\ \hline
$6.28593$ & $\mathrm{O}1^{-}\ \mathrm{O}2^{-}\ \mathrm{U}1^{-}\ \mathrm{O}3^{+}\ \mathrm{O}4^{+}\ \mathrm{U}5^{-}\ \mathrm{O}6^{-}\ \mathrm{O}5^{-}\ \mathrm{U}6^{-}\ \mathrm{U}4^{+}\ \mathrm{U}3^{+}\ \mathrm{U}2^{-}$ & $\mathrm{O}6^{-}\ \mathrm{U}6^{-}$ & $1$ & $[1,2]$ & $\leq 2$ \\ \hline
$6.28594$ & $\mathrm{O}1^{-}\ \mathrm{O}2^{-}\ \mathrm{U}1^{-}\ \mathrm{O}3^{+}\ \mathrm{O}4^{+}\ \mathrm{U}5^{-}\ \mathrm{O}6^{+}\ \mathrm{O}5^{-}\ \mathrm{U}6^{+}\ \mathrm{U}4^{+}\ \mathrm{U}3^{+}\ \mathrm{U}2^{-}$ & $\mathrm{O}6^{+}\ \mathrm{U}6^{+}$ & $1$ & $1$ & $1$ \\ \hline
$6.28595$ & $\mathrm{O}1^{-}\ \mathrm{O}2^{-}\ \mathrm{U}1^{-}\ \mathrm{O}3^{+}\ \mathrm{O}4^{+}\ \mathrm{U}5^{+}\ \mathrm{O}6^{-}\ \mathrm{O}5^{+}\ \mathrm{U}6^{-}\ \mathrm{U}4^{+}\ \mathrm{U}3^{+}\ \mathrm{U}2^{-}$ & $\mathrm{O}6^{-}\ \mathrm{U}6^{-}$ & $1$ & $1$ & $1$ \\ \hline
$6.33354$ & $\mathrm{O}1^{-}\ \mathrm{O}2^{+}\ \mathrm{U}1^{-}\ \mathrm{O}3^{-}\ \mathrm{U}4^{-}\ \mathrm{U}3^{-}\ \mathrm{O}4^{-}\ \mathrm{U}2^{+}\ \mathrm{O}5^{-}\ \mathrm{U}6^{-}\ \mathrm{U}5^{-}\ \mathrm{O}6^{-}$ & $\mathrm{U}4^{-}\ \mathrm{O}4^{-}\ \mathrm{U}6^{-}\ \mathrm{O}6^{-}$ & $1$ & $[1,2]$ & $\leq 2$ \\ \hline
$6.33355$ & $\mathrm{O}1^{-}\ \mathrm{O}2^{+}\ \mathrm{U}1^{-}\ \mathrm{O}3^{-}\ \mathrm{U}4^{-}\ \mathrm{U}3^{-}\ \mathrm{O}4^{-}\ \mathrm{U}2^{+}\ \mathrm{O}5^{-}\ \mathrm{U}6^{+}\ \mathrm{U}5^{-}\ \mathrm{O}6^{+}$ & $\mathrm{U}4^{-}\ \mathrm{O}4^{-}\ \mathrm{U}6^{+}\ \mathrm{O}6^{+}$ & $1$ & $1$ & $1$ \\ \hline
$6.33356$ & $\mathrm{O}1^{-}\ \mathrm{O}2^{+}\ \mathrm{U}1^{-}\ \mathrm{O}3^{-}\ \mathrm{U}4^{-}\ \mathrm{U}3^{-}\ \mathrm{O}4^{-}\ \mathrm{U}2^{+}\ \mathrm{O}5^{+}\ \mathrm{U}6^{-}\ \mathrm{U}5^{+}\ \mathrm{O}6^{-}$ & $\mathrm{U}4^{-}\ \mathrm{O}4^{-}\ \mathrm{U}6^{-}\ \mathrm{O}6^{-}$ & $1$ & $1$ & $1$ \\ \hline
$6.33357$ & $\mathrm{O}1^{-}\ \mathrm{O}2^{+}\ \mathrm{U}1^{-}\ \mathrm{O}3^{-}\ \mathrm{U}4^{-}\ \mathrm{U}3^{-}\ \mathrm{O}4^{-}\ \mathrm{U}2^{+}\ \mathrm{O}5^{+}\ \mathrm{U}6^{+}\ \mathrm{U}5^{+}\ \mathrm{O}6^{+}$ & $\mathrm{U}4^{-}\ \mathrm{O}4^{-}\ \mathrm{U}6^{+}\ \mathrm{O}6^{+}$ & $1$ & $1$ & $\leq 2$ \\ \hline
$6.33366$ & $\mathrm{O}1^{-}\ \mathrm{O}2^{+}\ \mathrm{U}1^{-}\ \mathrm{O}3^{+}\ \mathrm{U}4^{+}\ \mathrm{U}3^{+}\ \mathrm{O}4^{+}\ \mathrm{U}2^{+}\ \mathrm{O}5^{-}\ \mathrm{U}6^{-}\ \mathrm{U}5^{-}\ \mathrm{O}6^{-}$ & $\mathrm{U}4^{+}\ \mathrm{O}4^{+}\ \mathrm{U}6^{-}\ \mathrm{O}6^{-}$ & $1$ & $1$ & $\leq 2$ \\ \hline
$6.33367$ & $\mathrm{O}1^{-}\ \mathrm{O}2^{+}\ \mathrm{U}1^{-}\ \mathrm{O}3^{+}\ \mathrm{U}4^{+}\ \mathrm{U}3^{+}\ \mathrm{O}4^{+}\ \mathrm{U}2^{+}\ \mathrm{O}5^{-}\ \mathrm{U}6^{+}\ \mathrm{U}5^{-}\ \mathrm{O}6^{+}$ & $\mathrm{U}4^{+}\ \mathrm{O}4^{+}\ \mathrm{U}6^{+}\ \mathrm{O}6^{+}$ & $1$ & $1$ & $1$ \\ \hline
$6.33368$ & $\mathrm{O}1^{-}\ \mathrm{O}2^{+}\ \mathrm{U}1^{-}\ \mathrm{O}3^{+}\ \mathrm{U}4^{+}\ \mathrm{U}3^{+}\ \mathrm{O}4^{+}\ \mathrm{U}2^{+}\ \mathrm{O}5^{+}\ \mathrm{U}6^{-}\ \mathrm{U}5^{+}\ \mathrm{O}6^{-}$ & $\mathrm{U}4^{+}\ \mathrm{O}4^{+}\ \mathrm{U}6^{-}\ \mathrm{O}6^{-}$ & $1$ & $1$ & $1$ \\ \hline
$6.33369$ & $\mathrm{O}1^{-}\ \mathrm{O}2^{+}\ \mathrm{U}1^{-}\ \mathrm{O}3^{+}\ \mathrm{U}4^{+}\ \mathrm{U}3^{+}\ \mathrm{O}4^{+}\ \mathrm{U}2^{+}\ \mathrm{O}5^{+}\ \mathrm{U}6^{+}\ \mathrm{U}5^{+}\ \mathrm{O}6^{+}$ & $\mathrm{U}4^{+}\ \mathrm{O}4^{+}\ \mathrm{U}6^{+}\ \mathrm{O}6^{+}$ & $1$ & $[1,2]$ & $\leq 2$ \\ \hline
$6.42238$ & $\mathrm{O}1^{-}\ \mathrm{O}2^{-}\ \mathrm{U}1^{-}\ \mathrm{O}3^{-}\ \mathrm{U}4^{-}\ \mathrm{O}5^{-}\ \mathrm{U}6^{-}\ \mathrm{U}5^{-}\ \mathrm{O}6^{-}\ \mathrm{O}4^{-}\ \mathrm{U}3^{-}\ \mathrm{U}2^{-}$ & $\mathrm{U}6^{-}\ \mathrm{O}6^{-}$ & $2$ & $2$ & $2$ \\ \hline
$6.42239$ & $\mathrm{O}1^{-}\ \mathrm{O}2^{-}\ \mathrm{U}1^{-}\ \mathrm{O}3^{-}\ \mathrm{U}4^{-}\ \mathrm{O}5^{-}\ \mathrm{U}6^{+}\ \mathrm{U}5^{-}\ \mathrm{O}6^{+}\ \mathrm{O}4^{-}\ \mathrm{U}3^{-}\ \mathrm{U}2^{-}$ & $\mathrm{U}6^{+}\ \mathrm{O}6^{+}$ & $2$ & $1$ & $1$ \\ \hline
$6.42240$ & $\mathrm{O}1^{-}\ \mathrm{O}2^{-}\ \mathrm{U}1^{-}\ \mathrm{O}3^{-}\ \mathrm{U}4^{-}\ \mathrm{O}5^{+}\ \mathrm{U}6^{+}\ \mathrm{U}5^{+}\ \mathrm{O}6^{+}\ \mathrm{O}4^{-}\ \mathrm{U}3^{-}\ \mathrm{U}2^{-}$ & $\mathrm{U}6^{+}\ \mathrm{O}6^{+}$ & $2$ & $[1,2]$ & $\leq 2$ \\ \hline
$6.42241$ & $\mathrm{O}1^{-}\ \mathrm{O}2^{-}\ \mathrm{U}1^{-}\ \mathrm{O}3^{-}\ \mathrm{U}4^{+}\ \mathrm{O}5^{-}\ \mathrm{U}6^{-}\ \mathrm{U}5^{-}\ \mathrm{O}6^{-}\ \mathrm{O}4^{+}\ \mathrm{U}3^{-}\ \mathrm{U}2^{-}$ & $\mathrm{U}6^{-}\ \mathrm{O}6^{-}$ & $2$ & $[1,2]$ & $\leq 2$ \\ \hline
$6.42242$ & $\mathrm{O}1^{-}\ \mathrm{O}2^{-}\ \mathrm{U}1^{-}\ \mathrm{O}3^{-}\ \mathrm{U}4^{+}\ \mathrm{O}5^{+}\ \mathrm{U}6^{-}\ \mathrm{U}5^{+}\ \mathrm{O}6^{-}\ \mathrm{O}4^{+}\ \mathrm{U}3^{-}\ \mathrm{U}2^{-}$ & $\mathrm{U}6^{-}\ \mathrm{O}6^{-}$ & $2$ & $1$ & $1$ \\ \hline
$6.42243$ & $\mathrm{O}1^{-}\ \mathrm{O}2^{-}\ \mathrm{U}1^{-}\ \mathrm{O}3^{-}\ \mathrm{U}4^{+}\ \mathrm{O}5^{+}\ \mathrm{U}6^{+}\ \mathrm{U}5^{+}\ \mathrm{O}6^{+}\ \mathrm{O}4^{+}\ \mathrm{U}3^{-}\ \mathrm{U}2^{-}$ & $\mathrm{U}6^{+}\ \mathrm{O}6^{+}$ & $2$ & $[1,2]$ & $\leq 2$ \\ \hline
$6.42244$ & $\mathrm{O}1^{-}\ \mathrm{O}2^{-}\ \mathrm{U}1^{-}\ \mathrm{O}3^{+}\ \mathrm{U}4^{-}\ \mathrm{O}5^{-}\ \mathrm{U}6^{+}\ \mathrm{U}5^{-}\ \mathrm{O}6^{+}\ \mathrm{O}4^{-}\ \mathrm{U}3^{+}\ \mathrm{U}2^{-}$ & $\mathrm{U}6^{+}\ \mathrm{O}6^{+}$ & $2$ & $1$ & $1$ \\ \hline
$6.42245$ & $\mathrm{O}1^{-}\ \mathrm{O}2^{-}\ \mathrm{U}1^{-}\ \mathrm{O}3^{+}\ \mathrm{U}4^{-}\ \mathrm{O}5^{+}\ \mathrm{U}6^{+}\ \mathrm{U}5^{+}\ \mathrm{O}6^{+}\ \mathrm{O}4^{-}\ \mathrm{U}3^{+}\ \mathrm{U}2^{-}$ & $\mathrm{U}6^{+}\ \mathrm{O}6^{+}$ & $2$ & $[1,2]$ & $\leq 2$ \\ \hline
$6.42246$ & $\mathrm{O}1^{-}\ \mathrm{O}2^{-}\ \mathrm{U}1^{-}\ \mathrm{O}3^{+}\ \mathrm{U}4^{+}\ \mathrm{O}5^{-}\ \mathrm{U}6^{-}\ \mathrm{U}5^{-}\ \mathrm{O}6^{-}\ \mathrm{O}4^{+}\ \mathrm{U}3^{+}\ \mathrm{U}2^{-}$ & $\mathrm{U}6^{-}\ \mathrm{O}6^{-}$ & $2$ & $[1,2]$ & $\leq 2$ \\ \hline
$6.42247$ & $\mathrm{O}1^{-}\ \mathrm{O}2^{-}\ \mathrm{U}1^{-}\ \mathrm{O}3^{+}\ \mathrm{U}4^{+}\ \mathrm{O}5^{+}\ \mathrm{U}6^{-}\ \mathrm{U}5^{+}\ \mathrm{O}6^{-}\ \mathrm{O}4^{+}\ \mathrm{U}3^{+}\ \mathrm{U}2^{-}$ & $\mathrm{U}6^{-}\ \mathrm{O}6^{-}$ & $2$ & $1$ & $1$ \\ \hline
$6.46963$ & $\mathrm{O}1^{-}\ \mathrm{O}2^{-}\ \mathrm{U}1^{-}\ \mathrm{O}3^{-}\ \mathrm{U}4^{-}\ \mathrm{U}5^{-}\ \mathrm{O}6^{-}\ \mathrm{O}5^{-}\ \mathrm{U}6^{-}\ \mathrm{O}4^{-}\ \mathrm{U}3^{-}\ \mathrm{U}2^{-}$ & $\mathrm{O}6^{-}\ \mathrm{U}6^{-}$ & $2$ & $2$ & $2$ \\ \hline
$6.46964$ & $\mathrm{O}1^{-}\ \mathrm{O}2^{-}\ \mathrm{U}1^{-}\ \mathrm{O}3^{-}\ \mathrm{U}4^{-}\ \mathrm{U}5^{+}\ \mathrm{O}6^{-}\ \mathrm{O}5^{+}\ \mathrm{U}6^{-}\ \mathrm{O}4^{-}\ \mathrm{U}3^{-}\ \mathrm{U}2^{-}$ & $\mathrm{O}6^{-}\ \mathrm{U}6^{-}$ & $2$ & $1$ & $1$ \\ \hline
$6.46965$ & $\mathrm{O}1^{-}\ \mathrm{O}2^{-}\ \mathrm{U}1^{-}\ \mathrm{O}3^{-}\ \mathrm{U}4^{-}\ \mathrm{U}5^{+}\ \mathrm{O}6^{+}\ \mathrm{O}5^{+}\ \mathrm{U}6^{+}\ \mathrm{O}4^{-}\ \mathrm{U}3^{-}\ \mathrm{U}2^{-}$ & $\mathrm{O}6^{+}\ \mathrm{U}6^{+}$ & $2$ & $1$ & $\leq 2$ \\ \hline
$6.46966$ & $\mathrm{O}1^{-}\ \mathrm{O}2^{-}\ \mathrm{U}1^{-}\ \mathrm{O}3^{-}\ \mathrm{U}4^{+}\ \mathrm{U}5^{-}\ \mathrm{O}6^{-}\ \mathrm{O}5^{-}\ \mathrm{U}6^{-}\ \mathrm{O}4^{+}\ \mathrm{U}3^{-}\ \mathrm{U}2^{-}$ & $\mathrm{O}6^{-}\ \mathrm{U}6^{-}$ & $2$ & $[1,2]$ & $\leq 2$ \\ \hline
$6.46967$ & $\mathrm{O}1^{-}\ \mathrm{O}2^{-}\ \mathrm{U}1^{-}\ \mathrm{O}3^{-}\ \mathrm{U}4^{+}\ \mathrm{U}5^{-}\ \mathrm{O}6^{+}\ \mathrm{O}5^{-}\ \mathrm{U}6^{+}\ \mathrm{O}4^{+}\ \mathrm{U}3^{-}\ \mathrm{U}2^{-}$ & $\mathrm{O}6^{+}\ \mathrm{U}6^{+}$ & $2$ & $1$ & $1$ \\ \hline
$6.46968$ & $\mathrm{O}1^{-}\ \mathrm{O}2^{-}\ \mathrm{U}1^{-}\ \mathrm{O}3^{-}\ \mathrm{U}4^{+}\ \mathrm{U}5^{+}\ \mathrm{O}6^{+}\ \mathrm{O}5^{+}\ \mathrm{U}6^{+}\ \mathrm{O}4^{+}\ \mathrm{U}3^{-}\ \mathrm{U}2^{-}$ & $\mathrm{O}6^{+}\ \mathrm{U}6^{+}$ & $2$ & $1$ & $\leq 2$ \\ \hline
$6.46969$ & $\mathrm{O}1^{-}\ \mathrm{O}2^{-}\ \mathrm{U}1^{-}\ \mathrm{O}3^{+}\ \mathrm{U}4^{-}\ \mathrm{U}5^{+}\ \mathrm{O}6^{-}\ \mathrm{O}5^{+}\ \mathrm{U}6^{-}\ \mathrm{O}4^{-}\ \mathrm{U}3^{+}\ \mathrm{U}2^{-}$ & $\mathrm{O}6^{-}\ \mathrm{U}6^{-}$ & $2$ & $1$ & $1$ \\ \hline
$6.46970$ & $\mathrm{O}1^{-}\ \mathrm{O}2^{-}\ \mathrm{U}1^{-}\ \mathrm{O}3^{+}\ \mathrm{U}4^{-}\ \mathrm{U}5^{+}\ \mathrm{O}6^{+}\ \mathrm{O}5^{+}\ \mathrm{U}6^{+}\ \mathrm{O}4^{-}\ \mathrm{U}3^{+}\ \mathrm{U}2^{-}$ & $\mathrm{O}6^{+}\ \mathrm{U}6^{+}$ & $2$ & $1$ & $\leq 2$ \\ \hline
$6.46971$ & $\mathrm{O}1^{-}\ \mathrm{O}2^{-}\ \mathrm{U}1^{-}\ \mathrm{O}3^{+}\ \mathrm{U}4^{+}\ \mathrm{U}5^{-}\ \mathrm{O}6^{-}\ \mathrm{O}5^{-}\ \mathrm{U}6^{-}\ \mathrm{O}4^{+}\ \mathrm{U}3^{+}\ \mathrm{U}2^{-}$ & $\mathrm{O}6^{-}\ \mathrm{U}6^{-}$ & $2$ & $[1,2]$ & $\leq 2$ \\ \hline
$6.46972$ & $\mathrm{O}1^{-}\ \mathrm{O}2^{-}\ \mathrm{U}1^{-}\ \mathrm{O}3^{+}\ \mathrm{U}4^{+}\ \mathrm{U}5^{-}\ \mathrm{O}6^{+}\ \mathrm{O}5^{-}\ \mathrm{U}6^{+}\ \mathrm{O}4^{+}\ \mathrm{U}3^{+}\ \mathrm{U}2^{-}$ & $\mathrm{O}6^{+}\ \mathrm{U}6^{+}$ & $2$ & $1$ & $1$ \\ \hline
\end{longtable}
}
\section{Tool used}
AI assistance was strictly limited to coding syntax for the appendix table.
\bibliography{ref}
\bibliographystyle{abbrv}

\end{document}